\documentclass[11pt]{article}
\usepackage{amsfonts}
\usepackage{amsthm}
\usepackage{indentfirst}
\usepackage{graphics}
\usepackage{color}
\usepackage{cite}
\usepackage{latexsym}
\usepackage[paper=a4paper, left=2.4cm, right=2.4cm, top=2.5cm, bottom=1.8cm, headheight=5.5pt, footskip=1.0cm, footnotesep=0.8cm, centering, includefoot]{geometry}

\usepackage{amsmath}

\usepackage{todonotes}

\allowdisplaybreaks
\usepackage{amssymb}
\usepackage[dvips]{epsfig}
\usepackage{amscd}

\newtheorem{theorem}{Theorem}[section]
\newtheorem{remark}{Remark}[section]
\newtheorem{definition}{Definition}[section]
\newtheorem{lemma}[theorem]{Lemma}

\newtheorem{proposition}[theorem]{Proposition}

\newcommand{\vp}{\varphi}
\newcommand{\n}{\rho}

\newcommand{\ti}{\tilde}

\DeclareMathOperator{\loc}{loc}

\renewcommand{\div}{ {\rm div }  }

\newcommand{\pa}{\partial}
\renewcommand{\r}{\mathbb{R}}

\newcommand{\bt}{\begin{theorem}}
\newcommand{\bl}{\begin{lemma}}
\newcommand{\el}{\end{lemma}}
\newcommand{\et}{\end{theorem}}
\newcommand{\ga}{\gamma}

\newcommand{\al}{\alpha}
\newcommand{\ve}{\varepsilon}
\newcommand{\de}{\delta}
\newcommand{\la}{\label}

\newcommand{\bn}{\begin{eqnarray}}
\newcommand{\en}{\end{eqnarray}}
\newcommand{\bnn}{\begin{eqnarray*}}
\newcommand{\enn}{\end{eqnarray*}}

\newcommand{\bnnn}{\begin{eqnarray*}}
\newcommand{\ennn}{\end{eqnarray*}}
\newcommand{\ba}{\begin{aligned}}
\newcommand{\ea}{\end{aligned}}
\newcommand{\be}{\begin{equation}}
\newcommand{\ee}{\end{equation}}
\def\O{{\r^2 }}
\def\p{\partial}
\def\norm[#1]#2{\|#2\|_{#1}}
\def\lap{\triangle}

\def\o{\omega}

\newcommand{\si}{\sigma}

\def\la{\label}

\def\na{\nabla}
\def\pl{\partial}

\def\rr{\mathbb{R}^2}

\makeatletter
\@addtoreset{equation}{section}
\makeatother

\makeatletter
\@addtoreset{equation}{section}
\makeatother

\title{Global Well-posedness of Strong Solutions to the Cauchy Problem of 2D Nonhomogeneous Navier-Stokes Equations with Density-Dependent Viscosity and Vacuum}
\author{Bing Y{\small UAN}$^{a}$, Rong Z{\small HANG}$^{a,b}$, Peng Z{\small HOU}$^{c}$,\thanks{email: bingyuan@email.ncu.edu.cn (B. Yuan), rzhang0921@gmail.com (R. Zhang), zp000459@163.com (P. Zhou)} \\
{\normalsize a. School of Mathematics and Computer Sciences,}\\ {\normalsize  Nanchang University, Nanchang 330031, P. R. China;}\\
{\normalsize b. Institute of Mathematics and Interdisciplinary Sciences,}\\ {\normalsize  Nanchang University, Nanchang 330031, P. R. China;}\\
{\normalsize c. Jiangxi Flight University, Nanchang 330088, P. R. China;}}
\date{ }
\begin{document}
\maketitle
\begin{abstract}
This paper is concerned with the Cauchy problem for the modified two-dimensional (2D) nonhomogeneous incompressible Navier-Stokes equations with density-dependent viscosity. By fully using the structure of the system, 
we can obtain the key estimates of $\|\nabla \rho\|_{L_t^\infty L_x^q},q>2$ without any smallness asuumption on the initial data, and thus establish the global existence of the strong solutions with the far-field density being either vacuum or nonvacuum. Notably, the initial data can be arbitrarily large and the initial density is allowed to vanish. Furthermore, the large-time asymptotic behavior of the gradients of the velocity and the pressure is also established.

\end{abstract}

\section{Introduction}
In this paper, we consider the global solvability of the following 2D nonhomogeneous incompressible Navier-Stokes equations \cite{Z2008}:
\be\label{NS}
\begin{cases}
\partial_{t}\rho+\div(\rho u)=0, \\
\partial_{t}(\rho u )+\div(\rho u \otimes u )-\nabla^\bot\big(\mu(\rho)\omega\big)+\nabla P=0,\\
\div u=0,
\end{cases}
\ee
where $\rho=\rho(x,t)$, $u=(u^1,u^2)(x,t)$ and $P=P(x,t)$ represent the density, velocity and pressure of the fluid, respectively. $\o=\nabla^\perp\cdot u\triangleq\partial_2 u^1-\partial_1 u^2$ is the vorticity of the fluid and notation $\na^\perp$ stands for $\na^\perp f\triangleq(\partial_2 f, -\partial_1 f)$. The viscosity $\mu(\rho)$ satisfies the following hypothesis:
\begin{equation}\label{vis}
\displaystyle \mu \in C^1[0, \infty), \quad \mu(\rho)>0.
\end{equation}
Notice that as $\div u=0$, $\nabla^\bot\big(\mu(\rho)\omega\big)=\mu\Delta u$ when $\mu(\rho)=\mu=const$, so that the equation \eqref{NS} coincides with the classical nonhomogeneous, incompressible Navier-Stokes equation.

Let $\rho_\infty$ be a fixed nonnegative constant. We consider the Cauchy problem of \eqref{NS} with  $\left(\rho, u\right)$ satisfying the given initial data:
\begin{equation}\label{cztj}
	\rho(x,0)=\rho_0(x),\ \rho u (x,0)=\rho_0 u _0(x), \ x\in\mathbb{R}^2,
\end{equation}
and far-field behavior:
\begin{align}\label{ydxw}
	\displaystyle (\rho, u) \rightarrow (\rho_\infty, 0),\quad \text { as }|x| \rightarrow \infty.
\end{align}

According to \cite{L1996}, the more general density-dependent viscosity Navier-Stokes equations read as follows:
\be
\begin{cases}\label{rNS}
	\partial_{t}\rho+\div(\rho u)=0, \\
	\partial_{t}(\rho u )+\div(\rho u \otimes u )-\div\big(2\mu(\rho)\mathcal{D}(u)\big)+\nabla P=0,\\
	\div u=0,
\end{cases}
\ee
where $\mathcal{D}(u)=\frac12\big[\nabla u+(\nabla u)^T\big]$  is the deformation tensor. The system \eqref{NS} coincides with \eqref{rNS} provided that $\mu(\rho)$ is independent of $\rho$.

There is a lot of literature on the mathematical study of nonhomogeneous incompressible flow.
In particular, the system \eqref{NS} with constant viscosity has been
investigated extensively. When $\rho_0$ has a positive lower bound, Kazhikov \cite{K1974} established the global existence of weak solutions. Antontsev-Kazhikov-Monakhov \cite{AKM1990} gave the first result on local existence and uniqueness of strong solutions. Later, Abidi-Gui-Zhang \cite{Abi2011} proved the global well-posedness for the 3D Cauchy problem with small initial data in critical spaces. Abidi-Gui \cite{Abi2021} got the global well-posedness for the 2D Cauchy problem in critical Besov spaces and large initial data.
When $\rho_0$ is allowed to vanish, the global existence of weak solutions with finite energy is proved by Simon \cite{S1990}. Under some compatibility conditions, Choe-Kim \cite{CK2003} established the local existence of strong solutions for 3D bounded and unbounded domains. The global existence of strong solutions was obtained by L\"u-Shi-Zhong \cite{Lv2018} for 2D Cauchy problem with large initial data and Craig-Huang-Wang \cite{craig2013global} for the 3D Cauchy problem under some smallness conditions on the initial velocity.

When the viscosity is density-dependent, it is more difficult to investigate the global well-posedness of system \eqref{rNS} due to the strong coupling between viscosity coefficient and velocity. In fact, when the $\rho_0$ has a positive lower bound, Abidi-Zhang \cite{zhang5, zhang2015} obtained the global strong solutions in 2D or 3D under smallness conditions on $\|\mu(\rho_0)-1\|_{L^\infty}$, and then Paicu-Zhang \cite{pa2020} studied the problem for discontinuous initial data. Recently, Huang-Li-Zhang \cite{H-l-r} proved the global existence of strong solutions to 3D initial-boundary vaule problem provided the initial density is large enough. This is the first result concerning the well-posedness of large strong solution
to \eqref{rNS} in 3D without smallness of the velocity.
When the $\rho_0$ is allowed to vanish, Lions \cite{L1996} first obtained the global weak solutions with finite energy. Later, under the additional assumptions $\mu(\rho)$ is a small perturbation of a positive constant, Desjardins \cite{D1997} established the global weak solution with higher regularity for 2D case. Concerning the strong solutions, Huang-Wang \cite{HW2014} obtained the global strong solutions in 2D bounded domains when $\|\nabla\mu(\rho_0)\|_{L^q}(q>2)$ is small enough.
Huang-Wang \cite{HW2015} and Zhang \cite{Z2015} established the global strong solutions with small $\|\nabla u_0\|_{L^2}$ in 3D bounded domains.
For the 3D Cauchy problem, He-Li-L\"u \cite{He2021} proved the global well-posedness and exponential stability of strong solution under some smallness conditions on the initial velocity. However, this approach cannot extend to 2D. The key exponential decay-in-time in \cite{He2021} depends critically on the 3D Sobolev inequality, which fails in 2D whole space. Hence, the existence of global strong solutions with vacuum for the 2D Cauchy problem is still open, which is one of the main objectives of this paper.

It is noteworthy that all the global results for the strong solutions to the general density-dependent model \eqref{rNS} in particular for 2D case require some smallness assumptions on the velocity no matter for non-vacuum \cite{zhang5} or vacuum \cite{HW2014} cases. This contrasts sharply with the constant-viscosity case, where the global existence of strong solutions is obtained for any large data \cite{Abi2021,Lv2018}. Whether such large-data solutions can also be constructed in 2D for density-dependent case, and this forms another key motivation of our work. Concerning the revised model \eqref{NS}, Zhang \cite{Z2008} first obtained the global solution for large initial data by Littlewood-Paley theory when the initial density is strictly away from vacuum, then Liu \cite{Liu2019} reduced the regularity of $u_0$ to some extent. Thus, the question is raised: can global existence be established for the 2D Cauchy problem with vacuum and arbitrarily large initial data?
In this paper, we provide an affirmative answer to this question at least for system \eqref{NS}. More precisely, the main aim of this paper is to establish the global existence of strong solutions to the 2D Cauchy problem of \eqref{NS} with 
vacuum and large initial data.

Before formulating the main results, we first explain the notations and
conventions used throughout this paper. For integer $k\ge 1$ and $1\le p\le \infty$, we denote the standard Lebesgue and Sobolev spaces as follows:
\begin{equation}
\begin{cases}\nonumber
L^p=L^p(\r^2 ),\quad W^{k,p}= W^{k,p}(\r^2), \quad H^k = W^{k,2},\\
\|\cdot\|_{B_1 \cap B_2}=\|\cdot\|_{B_1}+\|\cdot\|_{B_2}, \text { for two Banach spaces } B_1 \text { and } B_2, \\
D^{k,p}=D^{k,p}(\r^2)=\{v\in L^1_{loc}(\r^2)|\nabla^kv\in L^p(\r^2)\},\\
C_{0,\sigma}^\infty=\{f\in C_0^\infty~|~\div f=0\},\quad D_{0,\sigma}^1=\overline{C_{0,\sigma}^\infty}~\mbox{closure~in~the~norm~of}~D^{1,2}.
\end{cases}
\end{equation}
For $R>0$, set
\begin{equation*}
B_R \triangleq\left.\left\{x\in\r^2\right|\,|x|<R \right\},
\quad \int f dx\triangleq\int_{\r^2}f dx.
\end{equation*}
Moreover, the material derivative of $f$, the upper bound of $\rho$ and
$\mu'$, the lower bound of $\mu$  are respectively defined by
$$
\dot{f}\triangleq f_t+u\cdot\nabla f,\quad \bar\rho \triangleq \max_{x\in \mathbb{R}^2} \rho_0,
\quad \bar{\mu'}\triangleq \max _{\rho \in[0, \bar{\rho}]} \mu'(\rho),\quad \underline{\mu} \triangleq \min _{\rho \in[0, \bar{\rho}]} \mu(\rho).
$$

Next, we give the definition of strong solutions to \eqref{NS}:
\begin{definition}\label{def1}
If all derivatives involved in \eqref{NS} for $(\rho,u,P)$ are regular distributions, and equations \eqref{NS} hold almost everywhere in $\mathbb{R}^2\times(0,T)$, then $(\rho,u,P)$ is called a strong solution to \eqref{NS}.
\end{definition}

Then, the first main result Theorem \ref{T1}, concerning the global existence of strong solution to the Cauchy problem \eqref{NS}-\eqref{ydxw} with vacuum state at far field, can be stated as follows:
\begin{theorem}\label{T1}
Let $\rho_\infty=0$, assume the initial data $(\rho_0\ge 0, u _0)$ satisfy that for some $a\in(1, 2)$ and $q>2$,
\begin{equation}\label{ct}
\rho_{0}\bar{x}^{a}\in L^{1}\cap H^{1}\cap W^{1,q},\ u_0\in D_{0,\sigma}^1,\ \sqrt{\rho_0} u _0\in L^2,
\end{equation}
where
\begin{equation}\label{xct}
\bar{x}\triangleq(e+|x|^2)^{1/2}\log^{1+\eta_0}(e+|x|^2),\quad \eta_0>0.
\end{equation}
Then the Cauchy problem \eqref{NS}-\eqref{ydxw} has a unique global strong solution $(\rho, u, P)$ satisfying that
for any $0<T<\infty$,
\begin{equation}\label{zzx}
\begin{cases}
\rho\in C([0,T];L^{1}\cap H^{1}\cap W^{1,q}), &\\
\rho\bar{x}^{a}\in L^{\infty}(0,T;L^{1}\cap H^{1}\cap W^{1,q}), &\\
\sqrt{\rho} u ,\ \nabla u ,\ \bar{x}^{-1} u ,\
\sqrt{t}\sqrt{\rho} u_t,~ \sqrt{t}\na P, ~\sqrt{t}\na^2 u\in L^{\infty}(0,T;L^2), &\\%
\nabla u \in L^{2}(0,T;H^1)\cap L^{(q+1)/q}(0,T;W^{1,q}), &\\
\na P\in  L^2(0,T;L^2)\cap  L^{(q+1)/q}(0,T;L^q), \\
\sqrt{t}\nabla u \in L^{2}(0,T;W^{1,q}),\ \sqrt{t}\nabla P \in L^{2}(0,T;L^q), &\\
\sqrt{\rho}u_t ,\ \sqrt{t}\nabla u_t ,\ \sqrt{t}\bar{x}^{-1}u_t\in L^{2}(\mathbb{R}^{2}\times(0,T)),
\end{cases}
\end{equation}
and
\begin{equation}\label{rjg}
\inf_{0\leq t\leq T}\int_{B_{N_{1}}}\rho(x,t)dx\geq\frac{1}{4},
\end{equation}
for some positive constant $N_1$. Moreover, $(\rho,u,P)$ has the following decay rates, that is, for any $p\in[2,\infty)$ and $t\geq1$, we have
\be \label{dsjxw}
t^\frac{p-1}{p}\|\nabla u(\cdot,t)\|_{L^p}+t\|\na\big(\mu(\rho) \o\big)(\cdot,t)\|_{L^2}+t\|\na P(\cdot,t)\|_{L^2}\le C,
\ee
where $C$ depends only on $p$, $\underline\mu, \bar{\mu'}, \|\n_0\|_{L^2\cap L^\infty}, \|\sqrt{\n_0} u_0\|_{L^2}$  and $\|\na u_0\|_{L^2} $.
\end{theorem}

For the case of the Cauchy problem \eqref{NS}-\eqref{ydxw} with the far-field density being away from vacuum, one can obtain the following global existence result:
\begin{theorem}\label{T2}
Let $\rho_\infty>0$, assume the initial data $(\rho_0\ge 0, u _0)$ satisfy that for some $r\in (1, \infty)$ and $q>2$,
\begin{equation}\label{2ct}
\rho_{0}-\rho_\infty\in L^r\cap L^\infty\cap D^{1,q},\ u_{0}\in D_{0,\sigma}^1,\ \sqrt{\rho_0} u _0\in L^2.
\end{equation}
Then the Cauchy problem \eqref{NS}-\eqref{ydxw} has a unique global strong solution $(\rho, u, P)$ satisfying that
for any $0<T<\infty$,
\begin{equation}\label{2zzx}
\begin{cases}
\rho-\rho_\infty\in C([0,T];L^{r}\cap L^\infty\cap D^{1,q}), &\\
\sqrt{\rho}u ,\ u ,\ \nabla u ,\
\sqrt{t} \sqrt{\rho}u_t,~ \sqrt{t}\na P,~ \sqrt{t}\na^2 u\in L^{\infty}(0,T;L^2), &\\
\nabla u \in L^{2}(0,T;H^1)\cap L^{(q+1)/q}(0,T;W^{1,q}), &\\
\na P\in  L^2(0,T;L^2)\cap  L^{(q+1)/q}(0,T;L^q), \\
\sqrt{t}\nabla u \in L^{2}(0,T;W^{1,q}),\ \sqrt{t}\nabla P \in L^{2}(0,T;L^q), &\\
\sqrt{\rho}u_t ,\ \sqrt{t}\nabla u_t ,\ \sqrt{t}u_t\in L^{2}(\mathbb{R}^{2}\times(0,T)).
\end{cases}
\end{equation}
Moreover, $(\rho,u,P)$ has the following decay rates, that is, for any $p\in[2,\infty)$ and $t\geq1$, we have
\be \la{2dsjxw}
t^\frac{p-1}{p}\|\nabla u(\cdot,t)\|_{L^p}+t\|\na\big(\mu(\rho) \o\big)(\cdot,t)\|_{L^2}+t\|\na P(\cdot,t)\|_{L^2}\le C,
\ee
where $C$ depends only on $p, r, \rho_\infty, \underline\mu, \bar{\mu'}, \|\n_0-\rho_\infty\|_{L^r\cap L^\infty}, \|\sqrt{\n_0} u_0\|_{L^2}$ and $\|\na u_0\|_{L^2} $.
\end{theorem}


\begin{remark}
It should be noted here that our Theorems \ref{T1} and \ref{T2} hold for any function $\mu(\rho)$ satisfying \eqref{vis} and for arbitrarily large initial density with vacuum (even has large area vacuum) without any smallness assumption only on the initial data, which is in sharp contrast to Zhang \cite{Z2008} and Liu\cite{Liu2019} where they need the initial density strictly away from vacuum, to  Abidi-Zhang \cite{zhang5} and Huang-Wang \cite{HW2014} where they need the smallness assumptions on the initial data.
\end{remark}

\begin{remark}
	Our results hold for constant viscosity  $\mu(\rho)=\mu$, thereby generalizes the earlier results of L\"u-Shi-Zhong\cite{Lv2018} where they consider the constant viscosity and require $\eta_0=1$ in \eqref{xct}.
\end{remark}

\begin{remark}
For the convenience of subsequent calculations, we consider the condition $\mu'(\rho)\in L^\infty(0,\bar\rho)$, which can be easily replaced by $\nabla\mu(\rho)\in L^q$ $(q>2)$ (see \cite{He2021,HW2014} for details). Under this more general assumption, we can choose $\mu(\rho)=1+\rho^\beta$ with $\beta\geq 0$, which has a broader range than the compressible case where the global large solution is obtained only when $\beta> 4/3$ (see \cite{Fan2022, Hua2016, Hua2022}).
\end{remark}
\begin{remark}
It should be mentioned here that the term $\nabla^\bot\big(\mu(\rho)\omega\big)$ is the $\div$-free part of the general diffusive term $\div\big(2\mu(\rho)\mathcal{D}(u)\big)$.
In fact, since $\div u=0$, then
\begin{align*}
	\div\big(2\mu(\rho)\mathcal{D} (u)\big) = \div\big(2\mu(\rho)\mathcal{A}(u)\big) + 2\div\big(  \mu(\rho) (\nabla u)^T\big) = \nabla^\bot\big(\mu(\rho)\omega\big) + 2 \nabla u \cdot \nabla \mu(\rho),
\end{align*}
where $\mathcal{A}(u)=\frac12\big[\nabla u-(\nabla u)^T\big]$ is the antisymmetric part of the deformation tensor.
In this paper, we only consider the div-free part of the diffusion term without dealing with the interaction of the gradient of density and the velocity. It would be interesting to study the global existence of equations \eqref{rNS} with vacuum, which is left for future.
\end{remark}

We now comment on the analysis of this paper. In fact, the local existence and uniqueness of strong solutions to the Cauchy problem \eqref{NS}-\eqref{ydxw} are guaranteed by Lemma \ref{lem21}. Thus, to extend the strong solutions globally in time and allow the density to vanish initially, one needs global a priori estimates, which are independent of the lower bound of the initial density, of smooth solutions to \eqref{NS}-\eqref{ydxw} in suitable higher norms. Motivated by Huang-Wang \cite{HW2014} recent studies on the blow-up criteria of strong solutions, it turns out that the key issue in this paper is to derive the upper bound of the $\|\nabla\rho\|_{L^q}(q>2)$, which is independent of the lower one of the initial density.


In the case of $\rho_\infty=0$, in order to establish the estimate of $\sup_{0\le t\le T}\|\nabla\rho\|_{L^q}$, the main difficulty is to obtain Lemmas \ref{lem3.3}-\ref{lem3.4}. On the one hand, motivated by L\"u-Shi-Zhong \cite{Lv2018}, we use the material derivatives of the velocity $\dot u$ instead of $u_t$ to avoid strong nonlinearity brought by the convective term  $\div(\rho u\otimes u)$. On the other hand, because the dissipative term in \eqref{NS}$_2$ involves density, estimates of $\nabla^2 u$ and $\nabla P$ generally depend on derivatives of $\rho$. To overcome this problem, we observe a special structure \eqref{a39}, which allows us to derive the substitute estimate
$$
\|\nabla\big(\mu(\rho)\o\big)\|_{L^2}+\|\nabla P\|_{L^2}\leq C\|\sqrt\rho\dot{u}\|_{L^2}.
$$
This inequality plays an important role in completing the proofs of Lemmas \ref{lem3.3}-\ref{lem3.4}.
Afterwards, we derive a spatially weighted mean estimate for the density.
Then, following the arguments in \cite{Fan2022,Hua2016} for the compressible Navier-Stokes equations and applying a Beale–Kato–Majda type inequality, we find that
$$
\begin{aligned}
\|\nabla u\|_{L^\infty}&\leq C\|\o\|_{L^\infty}\log(e+\|\na^2 u\|_{L^q})+C\|\na u\|_{L^2 } +C \\
&\leq C\|\mu(\rho)\o\|_{L^2}^{\frac{q-2}{2(q-1)}}\|\nabla\big(\mu(\rho)\o\big)\|_{L^q}^{\frac{q}{2(q-1)}}\log(e+\|\rho\dot u\|_{L^q}+\|\nabla\rho\|_{L^q}+\|\nabla u\|_{L^2})+C\\
&\leq C\|\rho\dot u\|_{L^q}^{\frac{q}{2(q-1)}}\log(e+\|\rho\dot u\|_{L^q}+\|\nabla\rho\|_{L^q})+C,
\end{aligned}
$$
the $\|\nabla\rho\|_{L^q}$ can be obtained via solving a logarithm Gr\"onwall inequality
 (see Proposition \ref{aupper}). This derivation also simultaneously provides a bound for $\|\nabla u\|_{L^1(0,T;L^\infty(\mathbb{R}^2))}$.

In the case of  $\rho_\infty>0$, the proof follows an argument similar to that in Section \ref{se13}. A key difference is that $\rho$ no longer possesses integrability, consequently, all estimates in Section \ref{se13} that rely on the integrability of $\rho$ must be adjusted. First, using the Gagliardo–Nirenberg inequality, we obtain the $\|u\|_{L^2}$ (see \eqref{vzv1}). Next, in the estimates of Lemma \ref{lem21.4}, the $\|P\|_{L^4}$ can no longer be used because estimate \eqref{l1} no longer holds. We therefore reformulate the term $J_3$ in Lemma \ref{lem3.4} as
\begin{equation}
	\begin{split}\nonumber
		\tilde{J}_3&=-\int\big(u_t^j \partial_{t}\partial_{j}P+u^k\partial_k u^j \partial_{t}\partial_{j}P +\dot{u}^j u^k\partial_k \partial_{j}P\big) dx\\
		&= -\frac{d}{dt}\int u^k\partial_k u^j \partial_{j}P dx+2\int \big(\dot u^k\partial_k u^j\partial_j P-u^i\partial_i u^k\partial_k u^j\partial_j P\big)dx\\
		&\quad+\int\big(\partial_{j}u^i\partial_i u^j u^k\partial_k P+\dot{u}^j \partial_{j}u^k\partial_k P\big)dx\\
		&\leq \frac{d}{dt}\int u^k\partial_k u^j \partial_{j}Pdx+C\|\nabla\dot{ u }\|_{L^2}\|\nabla u\|_{L^2}\|\nabla P\|_{L^2}+C\|u\|_{L^\infty}\|\nabla u\|_{L^4}^2\|\nabla P\|_{L^2},
	\end{split}
\end{equation}
which avoids dependence on $\|\rho\|_{L^\infty(0,T;L^2(\mathbb{R}^2))}$. Finally, similar to the proof of Proposition \ref{aupper}, we complete the proof of Proposition \ref{vvvt}.

The rest of the paper is organized as follows: In section \ref{sec2}, we collect some elementary facts and inequalities which will be used later. In Sections \ref{se13} and \ref{se321}, we derive the upper bound of the $\|\nabla\rho\|_{L^q}$, which is the key to extending the local solutions to all time, and some necessary lower and higher order estimates, when the constant states at far-field are vacuum ($\rho_\infty=0$) and nonvacuum ($\rho_\infty>0$), respectively. Finally, the main results Theorems \ref{T1} and \ref{T2} are proved in Section \ref{sec4}.
\section{Preliminaries}\label{sec2}
In this section, we shall enumerate some auxiliary lemmas.

To begin with, we recall the local existence result of strong solutions, whose proof for the case of $\rho_\infty=0$ was showed in \cite[Theorem 1.2]{LZ2015}, while the one for the case of $\rho_\infty>0$ can be derived in the similar standard method.
\begin{lemma}\label{lem21}
Assume that $(\rho_0, u _0)$ satisfies the conditions \eqref{ct} (if $\rho_\infty=0$) or \eqref{2ct}(if $\rho_\infty>0$). Then there exist a small time $T>0$ and a unique strong solution $(\rho\geq0, u, P)$ to the problem \eqref{NS}-\eqref{ydxw} in $\mathbb{R}^{2}\times(0,T)$ satisfying \eqref{zzx}-\eqref{dsjxw}(if $\rho_\infty=0$) or \eqref{2zzx}-\eqref{2dsjxw}(if $\rho_\infty>0$).
\end{lemma}

In what follows, we will use the convention that $C$ denotes a generic positive constant, and use $C(\al)$ to emphasize that $C$ depends on $\al.$

Next, the following Gagliardo-Nirenberg inequality (see \cite{N1959}) will be used later.
\begin{lemma}[Gagliardo-Nirenberg]\label{lem22}
For $p\in[2,\infty), s\in(1,\infty)$, and $r\in(2,\infty)$, there exists some generic constant $C>0$
such that for $f\in H^{1}(\mathbb{R}^2)$ and $g\in L^{s}(\mathbb{R}^2)\cap D^{1,r}(\mathbb{R}^2)$, we have
\begin{align}
&\|f\|_{L^p(\mathbb{R}^2)}^{p}\leq C(p)\|f\|_{L^2(\mathbb{R}^2)}^{2}\|\nabla f\|_{L^2(\mathbb{R}^2)}^{p-2}, \\
&\|g\|_{C(\overline{\mathbb{R}^2})}\leq C(s, r)\|g\|_{L^s(\mathbb{R}^2)}^{s(r-2)/(2r+s(r-2))}\|\nabla g\|_{L^r(\mathbb{R}^2)}^{2r/(2r+s(r-2))}.
\end{align}
\end{lemma}

By using equation $\Delta u=\nabla \div u+\nabla^\perp\nabla^\perp\cdot u$, we can easily obtain the following estimates.
\begin{lemma}\label{lem33}
For $1<p<\infty,$ there is a
constant  $C$ such that  the following estimate holds for all $\na u\in  L^p(\mathbb{R}^2)$ and $\na^2 v\in  L^p(\mathbb{R}^2)$,
\begin{align}\label{dumm}
\|\nabla u\|_{L^p(\mathbb{R}^2)} &\leq C(p)\left(\|\operatorname{div}u\|_{L^p(\mathbb{R}^2)}+\|\nabla^\perp\cdot u\|_{L^p(\mathbb{R}^2)}\right),\\
\|\nabla^2 v\|_{L^p(\mathbb{R}^2)} &\leq C(p)\left(\|\nabla\operatorname{div}v\|_{L^p(\mathbb{R}^2)}+\|\nabla^\perp\nabla^\perp\cdot v\|_{L^p(\mathbb{R}^2)}\right).\label{2dumm}
\end{align}
\end{lemma}


The following weighted $L^m$ bounds for elements of the Hilbert space $\tilde{D}^{1,2}(\O)\triangleq\{ v \in H_{\loc}^{1}(\mathbb{R}^2)|\nabla v \in L^{2}(\mathbb{R}^2)\}$ can be found in \cite[Theorem B.1]{L1996}.
\begin{lemma} \la{1leo}
For $m\in [2,\infty)$ and $\theta\in (1+m/2,\infty),$ there exists a positive constant $C$ such that for all $v\in  \tilde{D}^{1,2}(\O),$ \be\la{3h} \left(\int_{\O} \frac{|v|^m}{e+|x|^2}\left(\log \left(e+|x|^2\right)\right)^{-\theta}dx  \right)^{1/m}\le C(m, \theta)\|v\|_{L^2(B_1)}+C(m, \theta)\|\na v\|_{L^2(\O) }.\ee
\end{lemma}



A useful consequence of Lemma \ref{1leo} is the following weighted  bounds for elements of  $\ti D^{1,2}(\O) $ which   is important for our analysis.

\begin{lemma}[\cite{Hua2022}] \la{lemma2.6} Let  $\bar x$ and $\eta_0$ be as in \eqref{xct}. For  $\beta>1,$ assume that $\n \in L^1(\O)\cap L^\infty(\O)$ is a non-negative function such that
\be \la{2.12}   \int_{B_{N_1} }\n dx\ge M_1,  \quad \int \n^\beta dx\le M_2, \quad  \int \n \bar x^\alpha dx\le M_3,\ee
for positive constants $   M_i (i=1,\cdots,3),\alpha ,$ and $ N_1\ge 1.$  Then, for any $r\in (1,\infty)$, there is a positive constant $C$ depending only on   $ M_i (i=1,\cdots,3), N_1,\alpha, \beta, \eta_0$ and $r$ such that every $v\in \ti D^{1,2}(\O)  $ satisfies
 \be\la{2.16}\|\n v\|_{L^r}   \le C \left(  \|\n^{1/2} v\|_{L^2(B_{N_1})}
 + \|\na  v\|_{L^2}\right).\ee
\end{lemma}
\begin{proof}
First, for all $ \de\in (0,1),  \varepsilon\in(0,1)$, by direct calculations and combining with \eqref{3h} we obtain
\be\la{2.h10}\ba
&\|   v\bar x^{-\de} \|_{L^{(2+\varepsilon)/\de}}\\
&\leq C\|\na (v\bar x^{-\de})\|_{L^{2(2+\varepsilon)/(2+\varepsilon+2\delta)}}\\
&=C\|\bar   x^{-\de}\na v -\de  v\bar x^{-\de-1}\na\bar x \|_{L^{2(2+\varepsilon)/(2+\varepsilon+2\delta)}}\\
&\le C\|\na v\|_{L^2}\|\bar   x^{-\de}\|_{L^{(2+\varepsilon)/\de}}+C\| v\bar x^{-1}\log^{1+\eta_0}(e+|x|^2)\|_{L^2}\|\bar   x^{-\de}\|_{L^{(2+\varepsilon)/\de}}\\
&\le C \left(\|\na v\|_{L^2}+ \| v\|_{L^2(B_1)}\right).
\ea\ee

Next, it follows from   \eqref{2.12} and  the Poincar\'e type inequality     \cite[Lemma 3.2]{Fe}  that  there exists a positive constant $C $ depending only on $  M_1, M_2, N_1 ,$ and $\ga,$  such that  \be \la{3.12} \|v\|_{H^1(B_{ N_1} )}^2\le C \int_{B_{ N_1} }\n v^2dx +C \|\na v\|_{L^2(B_{ N_1} )}^2.\ee
This combined with  H\"older's inequality and \eqref{2.h10} yields that for  $r\in (1,\infty) $  and $\si=(2+\varepsilon)/(r(2+\varepsilon+\alpha)),$
 \bnn\ba \|\n v\|_{L^r}  &\le   \|(\n \bar x^\alpha)^\si\|_{L^{1/\si}} \| v\bar x^{-\alpha\si}\|_{L^{(2+\varepsilon)/(\alpha\si)}} \|\n\|_{L^\infty}^{1-\si}\\ &\le C  \left(  \|\n^{1/2} v\|_{L^2(B_{N_1})}
 + \|\na  v\|_{L^2}\right),\ea\enn
 which shows \eqref{2.16} and  finishes the proof of Lemma \ref{lemma2.6}.
\end{proof}

Next, let $\mathcal{H}^{1}(\mathbb{R}^2)$ and BMO$(\mathbb{R}^2)$ stand for the usual Hardy and BMO spaces (see \cite{S1993, CLMS1993}). Then the following well-known facts play a key role in the proof of Lemmas \ref{lem3.3}-\ref{lem3.4} in the next section.
\begin{lemma}\label{lem26}
(a) There is a positive constant $C$ such that
\begin{align}\label{mnxz}
\|E\cdot B\|_{\mathcal{H}^{1}(\O)}\leq C\|E\|_{L^{2}(\O)}\|B\|_{L^{2}(\O)},
\end{align}
for all $E\in L^{2}(\mathbb{R}^2)$ and $B\in L^{2}(\mathbb{R}^2)$ satisfying
\begin{equation*}
\div E=0,\ \nabla^{\bot}\cdot B=0\ \ \text{in}\ \ \mathcal{D}'(\mathbb{R}^2).
\end{equation*}
(b) There is a positive constant $C$ such that
\begin{equation}\label{lem1}
\| v \|_{{\rm BMO}(\O)}\leq C\|\nabla v \|_{L^{2}(\O)},
\end{equation} for all $  v \in \ti D^{1,2}(\mathbb{R}^2)$.
\end{lemma}

Then, we state the following Beale-Kato-Majda-type inequality
which was proved in \cite{B1,kato} when $\div u\equiv 0$ and will be used
later to estimate $\|\nabla u\|_{L^\infty}$ and
$\|\nabla\rho\|_{L^q}$.
\begin{lemma}
\la{le9}  For $2<q<\infty,$ there is a
constant  $C$ such that  the following estimate holds for all
$\na u\in  L^2(\O)\cap W^{1,q} (\O),$
$$
\|\na u\|_{L^\infty
}\le C(q)\left(\|{\rm div}u\|_{L^\infty }+ \|\nabla^\perp\cdot u\|_{L^\infty }
\right)\log(e+\|\na^2 u\|_{L^q })+C(q)\|\na u\|_{L^2 } +C(q).
$$
\end{lemma}

Finally, the following regularity results on the Stokes equations will be useful for our derivation of higher order a priori estimates, whose proof is motivated by some techniques due to \cite[Lemma 2.4]{He2021}.
\begin{lemma} \label{stokes}
For positive constants $\underline\mu$ and $q\in (2, \infty)$, in addition to \eqref{vis},  assume that $\mu(\rho)$ satisfies
$$
0<\underline{\mu} \leq \mu(\rho), \quad\nabla\mu(\rho)\in L^q.
$$
Then, if $F\in L^p$ with   $p\in(1, \infty)$, there exists some positive constant $C$ 
such that the unique weak solution $(u,P)$ to the following Cauchy problem
\be\label{3rd1}
\begin{cases}
 -\nabla^\perp\big(\mu(\rho)\o\big)  +\nabla P=F,\,\,\,\,&x\in \mathbb{R}^2,\\
 \div u=0,   \,\,\,&x\in \mathbb{R}^2,\\
u(x)\rightarrow0,\,\,\,\,&|x|\rightarrow\infty,
\end{cases}
\ee satisfies
\begin{align}
\|\na \big(\mu(\rho)\o\big)\|_{L^p}+\|\na P\|_{L^p}\le C(p)\|F\|_{L^p}.\label{2dumqa}
\end{align}
Moreover, if $F\in L^r$ with   $r\in[2q/(q+2), q],$ there exists some positive constant $C$ 
such that the unique weak solution $(u,P)$ satisfies
\begin{align}
\|\nabla^2 u\|_{L^r}\leq C(r, \underline\mu)\|F\|_{L^r}+C(r, \underline\mu, q)\|\nabla\mu(\rho)\|_{L^q}^\frac{2q(r-1)}{r(q-2)}\|\nabla u\|_{L^2}.\label{2duqa}
\end{align}
\end{lemma}
\begin{proof}
First, operating $\nabla^\perp\cdot$ and $\div$ to $\eqref{3rd1}_1$ yields that
\begin{align}\la{a39}
\lap \big(\mu(\rho)\omega\big)=-\na^\perp\cdot F,\quad \lap P=\div F,
\end{align}
which together with the standard $L^p$-estimate of elliptic equations implies that for $p\in (1,\infty),$
\be \la{muaa}
\|\na \big(\mu(\rho)\o\big)\|_{L^p}+\|\na P\|_{L^p}\le C(p)\|F\|_{L^p}.
\ee

Next, note that
$$
\nabla\omega=\nabla\left(\frac{\mu(\rho)\omega}{\mu(\rho)}\right)=\frac{1}{\mu(\rho)}\nabla\big({\mu(\rho)\omega}\big)
-\frac{1}{\mu(\rho)}\omega\nabla\mu(\rho),
$$
which, together with the Sobolev inequality and \eqref{muaa}, implies that for any $r\in[2q/(q+2), q]$,
$$
\begin{aligned}
\|\nabla\omega\|_{L^r}&\leq C(r, \underline\mu)\|\nabla\big({\mu(\rho)\omega}\big)\|_{L^r}+C(r, \underline\mu)\|\omega\nabla\mu(\rho)\|_{L^r}\\
&\leq C(r, \underline\mu)\|F\|_{L^r}+C(r, \underline\mu)\|\nabla\mu(\rho)\|_{L^q}\|\omega\|_{L^{\frac{rq}{q-r}}}\\
&\leq C(r, \underline\mu)\|F\|_{L^r}+C(r, \underline\mu, q)\|\nabla\mu(\rho)\|_{L^q}\|\omega\|_{L^2}^{1-\frac{rq+2r-2q}{2q(r-1)}}
\|\nabla\omega\|_{L^r}^\frac{rq+2r-2q}{2q(r-1)}\\
&\leq C(r, \underline\mu)\|F\|_{L^r}+C(r, \underline\mu, q)\|\nabla\mu(\rho)\|_{L^q}^\frac{2q(r-1)}{r(q-2)}\|\omega\|_{L^2}
+\frac{1}{2}
\|\nabla\omega\|_{L^r},
\end{aligned}
$$
which combined with \eqref{2dumm} leads to \eqref{2duqa}.
\end{proof}
\section{\la{se13}Vacuum at infinity ($\boldsymbol{\rho_\infty=0}$)}
In this section, we will establish some necessary a priori bounds for strong solutions $(\rho,u,P)$ to the Cauchy problem \eqref{NS}-\eqref{ydxw} in the case of $\rho_{\infty}=0$ to extend the local strong solution. Thus, let $T>0$ be a fixed time and $(\rho, u,P)$  be the strong solution to \eqref{NS}-\eqref{ydxw} on $\mathbb{R}^{2}\times(0,T]$ with initial data $(\rho_0,u_0)$ satisfying \eqref{ct}-\eqref{xct}.

In this section, in addition to the conditions of Theorem \ref{T1}, without loss of generality, we will always assume that $(\n_0,u_0)$  satisfies
\be\label{rho1}\frac12\le \int_{B_{N_0}}\rho_0(x)dx \le \int_{\r^2 }\n_0(x)dx \leq2,\ee
for some a positive constant $N_0.$


\subsection{\la{se3}A priori estimates (I): upper bound of the $\boldsymbol{\|\nabla\rho\|_{L^q}}$}

The following Proposition \ref{aupper}  will give an  upper bound of the $\|\nabla\rho\|_{L^q}$ which is the key to extending the local solution to be a global one.

\begin{proposition}\la{aupper}
	Under the conditions of Theorem \ref{T1},  for
	\bnn E_0\triangleq  \|\bar x^a\n_0\|_{L^1}+\|\rho_0\|_{L^\infty}+ \|\sqrt\n_0u_0 \|_{L^2}+ \|\na u_0\|_{L^2}+ \|\na \rho_0\|_{L^q}, \enn
	there is a positive constant  $C $ depending only on $a$, $q$, $\underline\mu$, $\bar{\mu'}$, $N_0$ and $E_0$ such that
	\be\la{b3.56}
	\sup_{t\in[0, T]}\|\nabla\rho\|_{L^q}+\int_0^T\|\na u\|_{L^\infty}dt\leq C(T).
	\ee
\end{proposition}

Before proving Proposition \ref{aupper}, we establish some a priori estimates, Lemmas \ref{lem3.1}-\ref{lem03.6}.

We begin with the following $L^p$-norm estimate of the density and the standard energy estimate.
\begin{lemma}\label{lem3.1}
There exists a positive constant $C$ depending only  on $\underline\mu$ and $\|\sqrt{\n_0} u_0\|_{L^2}$ such that
\begin{align}
&\sup_{t\in[0,T]}\|\rho\|_{L^p}\leq\|\rho_0\|_{L^p},\ for \ 1\leq p\leq\infty,\label{3.1}\\
&\sup_{t\in[0,T]}\|\sqrt{\rho} u \|_{L^2}^2
+\int_{0}^{T}\|\nabla u \|_{L^2}^{2}dt\leq C.\label{0.11}
\end{align}
\end{lemma}
\begin{proof}
First, standard arguments (see \cite{L1996}) directly give the desired \eqref{3.1}.

Next, multiplying \eqref{NS}$_2$ by $u$ and integrating (by parts) over ${\O}$, we derive that
$$
\frac{1}{2}\frac{d}{dt}\int\rho|u|^2dx+\int\mu(\rho)|\o|^2dx=0,
$$
which combined with  \eqref{dumm} gives \eqref{0.11}. The proof of Lemma \ref{lem3.1} is finished.
\end{proof}


Next, the following lemma concerns the key time-independent estimates on the $L^\infty(0,T;L^2)$-norm of the gradient of the velocity.
\begin{lemma}\label{lem3.3}
There exists a positive constant $C$ depending only  on $\underline\mu, \bar{\mu'}, \|\rho_0\|_{L^\infty}, \|\sqrt{\n_0} u_0\|_{L^2}$ and $\|\na u_0\|_{L^2}$
such that for $i=0,1,$
\begin{align}
\sup_{t\in[0,T]} \big(t^i\|\nabla u \|_{L^2}^2\big) +\int_{0}^{T}t^i\|\sqrt\rho\dot{ u
}\|_{L^2}^2dt\leq C.\label{3.5}
\end{align}
\end{lemma}
\begin{proof}
First, it follows from \eqref{NS}$_2$, \eqref{2dumqa} and \eqref{3.1} that
\begin{align}\label{paz}
\|\nabla\big(\mu(\rho)\o\big)\|_{L^2}+\|\nabla P\|_{L^2}\leq C\|\sqrt\rho\dot{u}\|_{L^2}.
\end{align}

Next, multiplying \eqref{NS}$_2$ by $\dot{u}$ and integrating the resulting equality over $\mathbb{R}^2$ lead to
\begin{equation}\label{3.6}
\int\rho|\dot{ u }|^{2}dx=\int\left(\nabla^\bot\big(\mu(\rho)\omega\big) \cdot\dot{ u }-\nabla P\cdot\dot{ u }\right)dx\triangleq I_{1}+I_{2}.
\end{equation}
It follows from integration by parts that
\begin{equation}
\begin{split}\nonumber
I_{1}  &  =\int\nabla^\bot\big(\mu(\rho)\omega\big) \cdot(u_t+ u \cdot\nabla u )dx \\
&=-\int\mu(\rho)\omega\o_t +\mu(\rho)\o \nabla^\perp\cdot (u\cdot\nabla u) dx \\
  &  = -\int\mu(\rho)\omega\o_t +\mu(\rho)\o u\cdot\nabla\o dx\\
 &  =-\frac12\frac{d}{dt}\int\mu(\rho)|\o|^2dx,
\end{split}
\end{equation}
where we have used $\na^\bot (u\cdot\na) \cdot u=0$.
The combination of \eqref{mnxz}, \eqref{lem1}, and \eqref{paz} yields that after using integration by parts
\begin{equation}\ba\label{zaqa}
I_{2}&=-\int\nabla P\cdot(u_t+u\cdot\nabla u)dx=\int P\div(u\cdot\nabla u)dx\\
&=\sum_{i=1}^2\int P\partial_{i}u\cdot\nabla u^{i}dx \leq C\sum_{i=1}^2\|P\|_{BMO}\|\partial_{i}u\cdot\nabla u^{i}\|_{\mathcal{H}^{1}}\\
&\leq C\|\nabla P\|_{L^2}\|\nabla u \|_{L^2}^{2}\leq \frac{1}{2}\|\sqrt\rho\dot{u}\|_{L^2}^2+C\|\nabla u\|_{L^2}^4,\ea
\end{equation}
where one has used the duality of $\mathcal{H}^1$ space and  BMO  one (see \cite[Chapter IV]{S1993}) in the first inequality.

Then, substituting $I_1$ and $I_2$ into \eqref{3.6}, and then using \eqref{dumm}, we obtain
\begin{align}\label{3.11}
\frac{d}{dt}\int\mu(\rho)|\o|^2dx+\|\sqrt{\rho}\dot{ u
}\|_{L^2}^{2}  \leq C\|\nabla u \|_{L^2}^{2}\int\mu(\rho)|\o|^2dx ,
\end{align}
multiplying \eqref{3.11} by $t^i (i=0,1)$, together with \eqref{0.11} and Gr\"onwall's inequality yields \eqref{3.5}. The proof of Lemma \ref{lem3.3} is finished.%
\end{proof}


\begin{lemma}\label{lem3.4}
There exists a positive constant $C$ depending only on $\underline\mu, \bar{\mu'}, \|\n_0\|_{L^2\cap L^\infty}, \|\sqrt{\n_0} u_0\|_{L^2}$ and $\|\na u_0\|_{L^2} $ such that for $i=1,2$,
\begin{equation}\label{133}
\sup_{t\in[0,T]}t^i\|\sqrt{\rho}\dot{ u }\|_{L^2}^2
+\int_{0}^{T}t^i\|\nabla\dot{ u }\|_{L^2}^{2}dt\leq C,
\end{equation}
and
\begin{equation}\label{i313}
\sup_{t\in[0,T]}\left(t^\frac{2(p-1)}{p}\|\nabla u\|_{L^p}^2+t^2\|\na\big(\mu(\rho)\o\big)\|_{L^2}^2+t^2\|\na P\|_{L^2}^2\right) \leq C(p).
\end{equation}
\end{lemma}
\begin{proof}
First, operating $\pa_t+
u\cdot\nabla $ to $ (\ref{NS})_2^j, j=1,2$   yields that
\begin{align}
&\p_t(\n   \dot u^1)+u\cdot\nabla(\n  \dot u^1)+\p_t\p_1P+u\cdot\nabla\p_1P=
\p_t\p_2\big(\mu(\rho)\o\big)+u\cdot\nabla\p_2\big(\mu(\rho)\o\big),\la{a4.6}\\
&\p_t(\n   \dot u^2)+u\cdot\nabla(\n  \dot u^2)+\p_t\p_2P+u\cdot\nabla\p_2P=
-\p_t\p_1\big(\mu(\rho)\o\big)-u\cdot\nabla\p_1\big(\mu(\rho)\o\big) .\la{a4.7}
\end{align}
Then, multiplying \eqref{a4.6} and \eqref{a4.7} by $\dot u^1$ and $\dot u^2$, respectively, combining those results together we obtain after integration by parts

\begin{equation}\label{3.15}
\begin{split}
\frac12\frac{d}{dt}\int\rho|\dot{ u }|^{2}dx&=\int\Big(\pl_t\pl_2\big(\mu(\rho)\omega\big)\dot{u}^1-\pl_t\pl_1\big(\mu(\rho)\omega\big)\dot{u}^2\Big) dx\\
&\quad+\int\Big( u\cdot\na\pl_2\big(\mu(\rho)\omega\big)\dot{u}^1-u\cdot\na\pl_1\big(\mu(\rho)\omega\big)\dot{u}^2\Big)dx\\
&\quad-\int\Big(\dot{u}^j \partial_{t}\partial_{j}P +\dot{u}^j u \cdot\nabla \partial_{j}P\Big) dx\triangleq\sum\limits_{i=1}^{3}J_i.
\end{split}
\end{equation}
It follows from integration by parts that
$$
\begin{aligned}
J_1 & =-\int\pl_t\big(\mu(\rho)\omega\big)\pl_2\dot{u}^1-\pl_t\big(\mu(\rho)\omega\big)\pl_1\dot{u}^2dx\\
&=-\int \pl_t\big(\mu(\rho)\omega\big)\left(\na^\bot\cdot\dot{u}\right) dx \\
& = -\int \pl_t\mu(\rho)\omega\na^\bot\cdot\dot{u}+\mu(\rho)\left(\na^\bot\cdot u_t\right)\na^\bot\cdot\dot{u}dx,
\end{aligned}
$$
and
$$
\begin{aligned}
J_2 & =\int -\pl_2u\cdot\na\big(\mu(\rho)\omega\big)\dot{u}^1-u\cdot\na\big(\mu(\rho)\omega\big)\pl_2\dot{u}^1
+\pl_1u\cdot\na\big(\mu(\rho)\omega\big)\dot{u}^2dx\\
&\quad+\int u\cdot\na\big(\mu(\rho)\omega\big)\pl_1\dot{u}^2dx \\
& = \int \mu(\rho)\omega\pl_2u\cdot \na\dot{u}^1-\mu(\rho)\omega\pl_1u \cdot \na\dot{u}^2dx-\int u\cdot\na\big(\mu(\rho)\omega\big)\na^\bot\cdot\dot{u}dx\\
& \leq C\|\nabla\dot{ u }\|_{L^2}\|\nabla u\|_{L^4}^2-\int u\cdot\na\mu(\rho)\omega\na^\bot\cdot\dot{u}
+\mu(\rho)\left(u\cdot\na\na^\bot\cdot u\right)\na^\bot\cdot\dot{u}dx\\
&\leq \varepsilon\|\nabla\dot{ u }\|_{L^2}^{2}+C_\varepsilon\|\nabla u\|_{L^4}^4-\int u\cdot\na\mu(\rho)\omega\na^\bot\cdot\dot{u}+\mu(\rho) \left(\na^\bot\cdot\big(u\cdot\na u\big)\right)\na^\bot\cdot\dot{u}dx,
\end{aligned}
$$
where we used $u\cdot\nabla\na^\bot\cdot u=\na^\bot\cdot(u\cdot\na u)$. Similar to article \cite{Lv2018} Lemma 3.3, we have
%
$$
\begin{aligned}
J_3 \leq & \frac{d}{dt}\int P\partial_{j}u^{i}\partial_{i}u^{j}dx
+C\int|P||\nabla\dot{ u }||\nabla u |dx
+C\int |P||\nabla u |^{3}dx \notag \\
\leq & \frac{d}{dt}\int P\partial_{j}u^{i}\partial_{i}u^{j}dx
+C_\varepsilon \left(\|P\|_{L^4}^{4}+\|\nabla u \|_{L^4}^{4}\right)
+\varepsilon\|\nabla\dot{ u }\|_{L^2}^{2}.
\end{aligned}
$$

Then making use of Lemma \ref{lem33}, we have
\begin{align}\label{plm}
\|\nabla \dot{u}\|_{L^2}\leq C\|\na^\bot\cdot\dot{u}\|_{L^2}+C\|\div\dot{u}\|_{L^2}\leq C\|\na^\bot\cdot\dot{u}\|_{L^2}+C\|\nabla u\|_{L^4}^2.
\end{align}
Substituting $J_1$-$J_3$ into \eqref{3.15}, choosing $\varepsilon$ is sufficiently small and combining with \eqref{plm} give
\be\ba\label{i47}
\Psi'(t)+C_0\int|\nabla\dot{ u }|^{2}dx\le C\|P\|_{L^4}^{4}+C\|\nabla u
\|_{L^4}^{4},
\ea\ee
where
\begin{equation*}
\Psi(t)\triangleq\frac12\int\rho|\dot{u}|^{2}dx
-\int P\partial_{j}u^{i}\partial_{i}u^{j}dx,
\end{equation*}
satisfies \be\la{psi1} \frac14\int\n |\dot u|^2dx-C\|\na u\|_{L^2}^4\le \Psi(t)\le \int\n |\dot u|^2dx+C\|\na u\|_{L^2}^4,\ee due to \eqref{paz} and \eqref{zaqa}.
Moreover, it follows from Sobolev's inequality, \eqref{2dumqa} and \eqref{3.1} that
\begin{equation}\label{l1}\ba
\|P\|_{L^4}^{4}+\|\nabla u \|_{L^4}^{4}
&\leq C\|P\|_{L^4}^{4}+C\|\mu(\rho)\omega \|_{L^4}^{4}\\
&\leq C\|\nabla P\|_{L^{4/3}}^{4}+ C\|\nabla\big(\mu(\rho)\o\big)\|_{L^{4/3}}^{4}\\
&\leq C\|\rho\dot{ u }\|_{L^{4/3}}^{4}
\leq C\|\rho\|_{L^2}^{2}\|\sqrt{\rho}\dot{ u }\|_{L^2}^{4}
\leq C\|\sqrt{\rho}\dot{ u }\|_{L^2}^{4}.\ea
\end{equation}
Multiplying \eqref{i47} by $t^i (i=1,2)$ and using \eqref{psi1}-\eqref{l1},  we obtain \eqref{133} from Gr\"onwall's inequality and \eqref{0.11}-\eqref{3.5}.

Finally, for any $p\in(2,\infty)$ we find
\begin{align}\label{zxza}
\|\nabla u\|_{L^p}\leq C\|\mu(\rho)\o\|_{L^p}\leq C\|\mu(\rho)\o\|_{L^2}^\frac{2}{p}\|\nabla\big(\mu(\rho)\o\big)\|_{L^2}^\frac{p-2}{p}
\leq C\|\nabla u\|_{L^2}^\frac{2}{p}\|\sqrt\rho\dot u\|_{L^2}^\frac{p-2}{p},
\end{align}
which together with \eqref{3.5}, \eqref{paz} and \eqref{133} yields \eqref{i313}. The proof of Lemma \ref{lem3.4} is finished.
\end{proof}
The following spatial weighted estimate on the density plays an important role in deriving the estimate $\|\nabla\rho\|_{L^q}$.
\begin{lemma}\label{lem03.6}
There exists a positive constant $C$ depending only on $a$, $q$, $\underline\mu$, $\bar{\mu'}$, $\|\rho_0\|_{L^\infty}$, $\|\n_0\bar x^a\|_{L^1}$, $\|\sqrt{\n_0} u_0\|_{L^2}$, $\|\na u_0\|_{L^2}$ and $N_0$ such that
\begin{equation}\label{06.1}
\sup_{t\in[0,T]}\|\rho\bar{x}^{a}\|_{L^{1}}+\int_{0}^{T}\left(\|\rho\dot{ u }\|_{L^q}^{\frac{q+1}{q}}
    +t\|\rho\dot{ u }\|_{L^q}^{2}\right)dt\leq C(T).
\end{equation}
\end{lemma}
\begin{proof}
First, multiplying $\eqref{NS}_1$ by $\bar x^a $ and integrating the result equality over $\O,$ we obtain after integration by parts and using \eqref{0.11} that
\bnn \ba  \frac{d}{dt}\int\n \bar x^a  dx &\le C\int \n |u|\bar x^{a -1}\log^{1+\eta_0}(e+|x|^2)dx\\ &\le C\left(\int\n \bar x^{2a -2} \log^{2(1+\eta_0)}(e+|x|^2)dx\right)^{1/2}\left(\int\n u^2 dx\right)^{1/2}\\ &\le C\left(\int\n \bar x^{a } dx\right)^{1/2},\ea\enn
which together with Gr\"onwall's inequality gives
\be\la{o3.7} \sup_{0\le t\le T}\int\n \bar x^a  dx\le C(T) .\ee

Next, for $N>1,$ let $\varphi_N$   be a smooth function such that
\be\nonumber
0\le\varphi_N(x)\le 1,\quad \varphi_N=
\begin{cases} 1& \mbox{ if }\,|x|\le N,\\ 0& \mbox{ if }\,|x|\ge 2N,
\end{cases} \quad |\na \varphi_N|\le 2N^{-1}.
\ee
It follows from $(\ref{NS})_1,$ \eqref{3.1} and \eqref{0.11} that
\bnn\ba
\frac{d}{dt}\int \n \vp_N dx &=\int \n u \cdot\na \vp_N dx\\
&\ge - 2N^{-1}\left(\int\n dx\right)^{1/2}\left(\int\n |u|^2dx\right)^{1/2}\ge -2\ti C^{1/2}N^{-1},
\ea\enn
which gives
\bnn
\inf\limits_{0\le t\le T}\int \n \vp_N dx\ge \int \n_0 \vp_N dx-2\ti C^{1/2}N^{-1}T.
\enn
This combined with \eqref{rho1} yields that for  $N_1\triangleq 2(2+ N_0+8\ti C^{1/2}T)$,
\begin{align}\label{zzaz}
\inf\limits_{0\le t\le T}\int_{B_{N_1}  }\n  dx\ge \int \n \vp_{N_1/2}dx\ge 1/4.
\end{align}

Finally, it follows from H\"older inequality, \eqref{2.16}, \eqref{3.1}, \eqref{o3.7} and \eqref{zzaz} that
\begin{align}\label{4.6}
\|\rho\dot{ u }\|_{L^q}
 &  \leq C\|\rho\dot{ u }\|_{L^2}^{\frac{2(q-1)}{q^2-2}}
\|\rho\dot{ u }\|_{L^{q^2}}^{\frac{q(q-2)}{q^2-2}} \notag \\
 &  \leq C(T)\|\rho\dot{ u }\|_{L^2}^{\frac{2(q-1)}{q^2-2}}
\left(\|\sqrt{\rho}\dot{ u }\|_{L^2}+\|\nabla\dot{ u }\|_{L^2}\right)^{\frac{q(q-2)}{q^2-2}}  \\
 &  \leq C(T)\left(\|\sqrt{\rho}\dot{ u }\|_{L^2}
+\|\sqrt{\rho}\dot{ u }\|_{L^2}^{\frac{2(q-1)}{q^2-2}}
\|\nabla\dot{ u }\|_{L^2}^{\frac{q(q-2)}{q^2-2}}\right),\notag
\end{align}
which together with \eqref{3.5} and \eqref{133} leads to
\begin{equation}
\begin{split}\nonumber
 &  \int_{0}^{T}\left(\|\rho\dot{ u }\|_{L^q}^{\frac{q+1}{q}}
    +t\|\rho\dot{ u }\|_{L^q}^{2}\right)dt \\
 &  \leq C(T)\int_{0}^{T}\left(\|\sqrt{\rho}\dot{ u }\|_{L^2}^{2}
    +t\|\nabla\dot{ u }\|_{L^2}^{2}
    +t^{-\frac{q^{3}-q^{2}-2q-1}{q^{3}-q^{2}-2q}}+1\right)dt  \\
 &  \leq C(T).
\end{split}
\end{equation}
This finishes the proof of Lemma \ref{lem03.6}.
 \end{proof}

Now we are in a position to prove Proposition \ref{aupper}.

{\it\bf Proof of Proposition \ref{aupper}.}
First, one gets from Beale-Kato-Majda-type inequality, \eqref{2dumqa}, \eqref{2duqa} and \eqref{3.5} that 
\begin{equation}
\begin{split}\label{zaq}
&\|\na u\|_{L^\infty}\\
&\le C\left(\|{\rm div}u\|_{L^\infty }+ \|\o\|_{L^\infty }
\right)\log(e+\|\na^2 u\|_{L^q})+C\|\na u\|_{L^2 } +C \\
&\leq C\|\mu(\rho)\o\|_{L^2}^{\frac{q-2}{2(q-1)}}\|\nabla\big(\mu(\rho)\o\big)\|_{L^q}^{\frac{q}{2(q-1)}}\log(e+\|\rho\dot u\|_{L^q}+\|\nabla\rho\|_{L^q}+\|\nabla u\|_{L^2})+C\\
&\leq C\|\rho\dot u\|_{L^q}^{\frac{q}{2(q-1)}}\log(e+\|\rho\dot u\|_{L^q}+\|\nabla\rho\|_{L^q})+C.
\end{split}
\end{equation}

Next, it follows from the mass equation \eqref{NS}$_1$ that $\na\n$ satisfies 
\begin{equation}
\begin{split}\label{zaq1}
&\frac{d}{dt}\log(e+\|\nabla\rho\|_{L^q})\\
&\leq \|\nabla u\|_{L^\infty}\\
&\leq C\|\rho\dot u\|_{L^q}^{\frac{q}{2(q-1)}}\log(e+\|\rho\dot u\|_{L^q}+\|\nabla\rho\|_{L^q})+C\\
&\leq C\|\rho\dot u\|_{L^q}^{\frac{q}{2(q-1)}}\log(e+\|\nabla\rho\|_{L^q})+C\|\rho\dot u\|_{L^q}^{\frac{q}{2(q-1)}}\log(e+\|\rho\dot u\|_{L^q})+C,
\end{split}
\end{equation}
it from \eqref{06.1} implies
\be \la{nu1}\int_0^T\left(\|\rho\dot u\|_{L^q}^{\frac{q}{2(q-1)}}+\|\rho\dot u\|_{L^q}^{\frac{q}{2(q-1)}}\log(e+\|\rho\dot u\|_{L^q})\right)dt\le C(T), \nonumber\ee
which together with \eqref{zaq1} and Gr\"onwall's inequality lead to
$$\sup_{t\in[0, T]}\|\nabla\rho\|_{L^q}\leq C(T).$$
Furthermore, which combining with \eqref{zaq} yields
$$
\int_0^T\|\na u\|_{L^\infty}dt\le C(T).
$$
We complete the proof of Proposition \ref{aupper}.
\subsection{\la{se4}A priori estimates (II): higher order estimates}
\begin{lemma}\label{lem3.5}
There exists a positive constant $C$ depending only on $a$, $q$, $\underline\mu$, $\bar{\mu'}$, $\|\rho_0\|_{L^1\cap H^1\cap W^{1,q}}$, $\|\n_0\bar x^a\|_{L^1}$, $\|\sqrt{\n_0} u_0\|_{L^2}$, $\|\na u_0\|_{L^2}$ and $N_0$ such that
\begin{equation}
\begin{split}\label{16.1}
\sup_{t\in[0,T]}\|\rho\|_{H^{1}}
&+\int_{0}^{T}\left(\|\nabla^{2} u \|_{L^2}^2+\|\nabla^{2} u \|_{L^q}^{\frac{q+1}{q}}
+t\|\nabla^{2} u \|_{L^2 \cap L^q}^2\right)dt \\
&+\int_{0}^{T}\left(\|\nabla P\|_{L^2}^2+\|\nabla P \|_{L^q}^\frac{q+1}{q}+t\|\nabla P\|_{L^2\cap L^q}^2\right)dt
\leq C(T).
\end{split}
\end{equation}
\end{lemma}
\begin{proof}
First, in a similar way to \eqref{zaq1}, we have
\begin{equation}\label{4.8}
\sup_{t\in[0,T]}\|\nabla\rho\|_{L^2}\leq C.
\end{equation}

Next, it follows from \eqref{NS}$_2$, \eqref{2dumqa}, \eqref{2duqa}, \eqref{3.1} and \eqref{3.5} that
\begin{equation}\label{paztg}
\begin{split}
\|\nabla^2 u\|_{L^2}+\|\nabla P\|_{L^2}&\leq C\|\sqrt\rho\dot{u}\|_{L^2}+C\|\nabla u\|_{L^2}\leq C\|\sqrt\rho\dot{u}\|_{L^2}+C,\\
\|\nabla^2 u\|_{L^q}+\|\nabla P\|_{L^q}&\leq C\|\sqrt\rho\dot{u}\|_{L^q}+C\|\nabla u\|_{L^2}\leq C\|\sqrt\rho\dot{u}\|_{L^q}+C,
\end{split}
\end{equation}
which together with \eqref{3.5} and \eqref{06.1} yields that
$$
\begin{aligned}
& \int_{0}^{T}\left(\|\nabla^2 u \|_{L^2}^2+\|\nabla^2 u \|_{L^q}^\frac{q+1}{q}+t\|\nabla^{2} u \|_{L^2 \cap L^q}^2\right)dt \notag \\
&\quad+\int_0^T\left(\|\nabla P\|_{L^2}^2+\|\nabla P \|_{L^q}^\frac{q+1}{q}+t\|\nabla P\|_{L^2\cap L^q}^2\right)dt\leq C.
\end{aligned}
$$
Which combining with \eqref{3.1}, \eqref{4.8} gives \eqref{16.1} and completes the proof of Lemma \ref{lem3.5}.
\end{proof}
\begin{lemma}\label{lem3.6}
There exists a positive constant $C$ depending only on $a$, $q$, $\underline\mu$, $\bar{\mu'}$, $\|\n_0\bar x^a\|_{L^1\cap H^1\cap W^{1,q}}$, $\|\sqrt{\n_0} u_0\|_{L^2}$, $\|\na u_0\|_{L^2}$ and $N_0$ such that
\begin{equation}\label{6.1}
\sup_{t\in[0,T]}\|\rho\bar{x}^{a}\|_{L^1\cap H^1\cap W^{1,q}}\leq C(T).
\end{equation}
\end{lemma}
\begin{proof}
First, it follows from \eqref{2.h10}, \eqref{3.12}, \eqref{3.1}-\eqref{3.5} and \eqref{zzaz} that for any $\ve\in(0,1)$ and $\de\in(0,1)$
 \be\la{a4.21} \|u\bar x^{-\de} \|_{L^{(2+\varepsilon)/\de}}\le C(\de, \varepsilon)\|\sqrt\rho u\|_{L^2}+C(\de, \varepsilon)\|\nabla u\|_{L^2}\le C(\de, \varepsilon).\ee
Direct calculations show that for $q>2$,
\bnn \ba  \|\na (u\bar x^{-\de})\|_{L^q}&\le C(\de)\|\na u\|_{L^q}+C(\de)
\|u \bar x^{-\delta}\|_{L^\infty}\|(e+|x|^2)^{-1/2}\|_{L^q}\\ &\le
C(\de)\|\na u\|_{L^q}+\frac12\|\na (u\bar x^{-\de})\|_{L^q}+C(\de)\|u\bar x^{-\de}\|_{L^{2.5/\de}},\ea\enn
which combined with \eqref{a4.21} implies
\be \la{a4.22}\|u \bar x^{-\de}\|_{L^\infty}\le C(\de) + C(\de)  \|\na u\|_{L^q} . \ee

Then, one derives from \eqref{NS}$_1$ that $v\triangleq\rho\bar{x}^a$ satisfies
\begin{equation}\label{A}
\partial_t v+ u \cdot\nabla v-av u \cdot\nabla\log\bar{x}=0.
\end{equation}
Taking the $x_i$-derivative on the both side of \eqref{A} gives
\begin{align}\label{B}
0=  &  \partial_{t}\partial_{i}v+ u \cdot\nabla\partial_{i}v
+\partial_{i} u \cdot\nabla v
-a\partial_{i}v u \cdot\nabla\log\bar{x} \notag \\
 &  -av\partial_{i} u \cdot\nabla\log\bar{x}
-av u \cdot\partial_{i}\nabla\log\bar{x}.
\end{align}
For any $r\in[2, q]$, multiplying \eqref{B} by $|\na v|^{r-2}\partial_{i}v$ and integrating the resulting equality over $\mathbb{R}^2$, we obtain after integration by parts that
\be\la{6.4}\ba
(\|\na v\|_{L^r} )_t& \le C(\|\na u\|_{L^\infty}+\|u\cdot \na \log \bar x\|_{L^\infty}) \|\na v\|_{L^r} \\
&\quad +C\|v\|_{L^\infty}\left( \||\na u||\na\log \bar x|\|_{L^r}+\||  u||\na^2\log \bar x|\|_{L^r}\right)\\
& \le C(1 +\|\na u\|_{W^{1,q}})  \|\na v\|_{L^r} \\
&\quad+C\|v\|_{L^\infty}\left(\|\na u\|_{L^r}+\|u\bar x^{-1/4}\|_{L^\infty}\|\bar x^{-3/2}\|_{L^r}\right) \\
& \le C(1 +\|\na^2u\|_{L^r}+\|\na u\|_{W^{1,q}})(1+ \|\na v\|_{L^r}+\|\na v\|_{L^q}),
\ea\ee
where in the second and the last inequalities, one has used \eqref{06.1} and \eqref{a4.22}.
Choosing $r=q$ in \eqref{6.4}, together with \eqref{16.1} thus shows
\begin{equation}\label{6.5}
\sup_{t\in[0,T]}\|\nabla(\rho\bar{x}^a)\|_{L^q}\leq C.
\end{equation}
Setting $r=2$ in \eqref{6.4}, we deduce from \eqref{16.1} and \eqref{6.5} that
\begin{equation}\nonumber
\sup_{t\in[0,T]}\|\nabla(\rho\bar{x}^a)\|_{L^2}\leq C.
\end{equation}
This combined with \eqref{6.5}, \eqref{06.1} gives \eqref{6.1} and finishes the proof of Lemma \ref{lem3.6}.   
\end{proof}

\begin{lemma}\label{lem3.8}
There exists a positive constant $C$ depending only on $a$, $q$, $\underline\mu$, $\bar{\mu'}$, $\|\n_0\bar x^a\|_{L^1\cap H^1\cap W^{1,q}}$, $\|\sqrt{\n_0} u_0\|_{L^2}$, $\|\na u_0\|_{L^2}$ and $N_0$ such that
\begin{equation}\label{7.1}
\sup_{t\in[0,T]}t\|\sqrt{\rho}u_t \|_{L^2}^{2}
+\int_{0}^{T}t\|\nabla u_t \|_{L^2}^2dt\leq C(T).
\end{equation}
\end{lemma}
\begin{proof}
First, it follows from  \eqref{3h}, \eqref{2.h10}, \eqref{3.12}, \eqref{3.1} and \eqref {zzaz}
 that for $\ve\in(0,1)$ and $\eta>0,$ every $v\in \tilde D^{1,2}(\rr)$ satisfies
\begin{equation} \la{5.d1}
\begin{split}
\|\n^\eta v\|_{L^{(2+\ve)/\tilde\eta}}
&\leq \|\n^\eta {\bar x}^{a\eta}\|_{L^{2(2+\ve)/\tilde\eta}}\|v {\bar x}^{-a\eta}\|_{L^{2(2+\ve)/\tilde\eta}}\\
&\le C(\ve,\eta)\|\n^{1/2}v\|_{L^2}+C(\ve,\eta)\|\na v\|_{L^2},
\end{split}
\end{equation}
with $\tilde\eta=\min\{1,\eta\}$. This combined with \eqref{3.5}, \eqref{6.1} and \eqref{a4.22}  yields that
\begin{equation}
\begin{split}  \la{5.d2}
\|\n^\eta u\|_{L^{(2+\ve)/\tilde\eta}}\le C(\ve,\eta),\quad
\|\n^\eta u\|_{L^\infty}\le C(\ve,\eta)\|\nabla u\|_{H^1}\le C(\ve,\eta)(\|\sqrt\rho\dot u\|_{L^2}+1).
\end{split}
\end{equation}

Next, differentiating \eqref{NS}$_2$ with respect to $t$ leads to
\begin{equation}\label{7.2}
\rho u_{tt}+\rho u \cdot\nabla u_t-\nabla^\perp\big(\mu(\rho)\o_t\big)
+\nabla P_t=-\rho_{t} u_t-(\rho u )_{t}\cdot\nabla u -\nabla^\perp\big(u\cdot\nabla\mu(\rho)\o\big).
\end{equation}
Multiplying \eqref{7.2} by $u_t$ and integrating the resulting equation by parts over $\mathbb{R}^2$,
we obtain after using \eqref{NS}$_1$ and \eqref{NS}$_3$ that
\begin{equation}\label{7.3}
\begin{split}
&\frac{1}{2}\frac{d}{dt}\|\sqrt{\rho} u_t \|_{L^2}^{2}
+\int\mu(\rho)|\o_t|^2dx\\
& =-\int \rho_{t}|u_t |^2dx
-\int (\rho u )_{t}\cdot\nabla u \cdot u_t dx-\int \nabla^\perp\big(u\cdot\nabla\mu(\rho)\o\big)\cdot u_tdx\\
&  \leq C\int\rho| u ||u_t |\big(|\nabla u_t|+|\nabla u |^2
+| u ||\nabla^2 u |\big)dx  \\
& \quad +C\int\rho| u |^2|\nabla u||\nabla u_t|dx+C\int\rho|u_t|^2|\nabla u |dx+C\int|u\cdot\nabla\rho||\omega\omega_t |dx \\
& \triangleq K_1+K_2+K_3+K_4.
\end{split}
\end{equation}

Now, we will use the Garliardo-Nirenberg inequality, \eqref{dumm}, \eqref{3.5}, \eqref{6.1}-\eqref{a4.21} and \eqref{5.d1}-\eqref{5.d2} to estimate each term on the right-hand side of \eqref{7.3} as follows:
$$
\begin{aligned}
K_1
&\leq C\|\sqrt{\rho} u \|_{L^6}\|\sqrt{\rho}u_t\|_{L^2}^{1/2}\|\sqrt{\rho}u_t\|_{L^6}^{1/2}\left(\|\nabla u_t\|_{L^2}+\|\nabla u \|_{L^4}^2\right) \notag \\
&\quad +C\|\rho^{1/4} u \|_{L^{12}}^{2}\|\sqrt{\rho}u_t\|_{L^2}^{1/2}\|\sqrt{\rho}u_t\|_{L^6}^{1/2}\|\nabla^2 u \|_{L^2}
 \notag \\
&\leq C\|\sqrt{\rho}u_t\|_{L^2}^{1/2}\left(\|\sqrt\rho u_t \|_{L^2}+\|\nabla u_t\|_{L^2}\right)^{1/2}\left(\|\nabla u_t\|_{L^2}+\|\nabla^2 u \|_{L^2}\right)\notag \\
&\leq \frac{1}{6}\|\sqrt{\mu(\rho)}\o_t\|_{L^2}^{2}+C\|\sqrt{\rho}u_t\|_{L^2}^{2}+C\|\nabla^2 u \|_{L^2}^2,
\end{aligned}
$$
$$
\begin{aligned}
K_2+K_3
&  \leq C \|\n^{1/2} u\|_{L^8}^{2}\|\na u\|_{L^4} \| \na u_{t}\|_{L^{2}}+ C\|\nabla u \|_{L^2}\|\sqrt{\rho} u_t\|_{L^6}^{3/2}
\|\sqrt{\rho}u_t\|_{L^2}^{1/2} \notag \\
&  \leq \frac{1}{6}\|\sqrt{\mu(\rho)}\o_t\|_{L^2}^{2}
+C\|\sqrt{\rho}u_t\|_{L^2}^{2}+C\|\nabla^2 u \|_{L^2}^2+C,
\end{aligned}
$$
and
$$
\begin{aligned}
 K_4 &  \leq C\|u{\bar x}^{-a} \|_{L^{4q/(q-2)}}\|\na\rho{\bar x}^{a}\|_{L^q}
\|\o\|_{L^{4q/(q-2)}}\|\o_t\|_{L^2} \notag \\
&  \leq C(1+\|\nabla^2 u\|_{L^2})\|\o_t\|_{L^2}\notag \\
&  \leq  \frac{1}{6}\|\sqrt{\mu(\rho)}\o_t\|_{L^2}^{2}+C\|\nabla^2 u \|_{L^2}^2+C.
\end{aligned}
$$

Substituting $K_1$-$K_4$ into \eqref{7.3} gives
\begin{equation}\label{11.3}
\frac{d}{dt}\|\sqrt{\rho}u_t\|_{L^2}^{2}
+\int\mu(\rho)|\o_t|^2dx\leq C\|\sqrt{\rho}u_t\|_{L^2}^{2}+C\|\nabla^2 u \|_{L^2}^2+C.
\end{equation}

Finally, it follows from \eqref{3.5}, \eqref{5.d2} and Garliardo-Nirenberg inequality that
\begin{align}\label{jianew}
\int_0^T \|\sqrt{\rho} u_t\|_{L^2}^2dt&\le \int_0^T \left(\|\sqrt{\rho}\dot u\|_{L^2}^2+\|\sqrt{\rho}|u||\na u|\|_{L^2}^2\right)dt \notag \\
&\le C+C\int_0^T \|\sqrt{\rho} u \|_{L^\infty}^2\|\na u\|_{L^2}^2 dt  \\
&\le C,\notag
\end{align}
then multiplying \eqref{11.3} by $t$, using Gr\"onwall's inequality, \eqref{16.1} and \eqref{jianew}, we obtain \eqref{7.1} and complete the proof of Lemma \ref{lem3.8}. 
\end{proof}
\section{\la{se321}Nonvacuum at infinity ($\boldsymbol{\rho_\infty>0}$)}
In this section, we aim at the case of $\rho_{\infty}>0$. Analogue to Section \ref{se13}, we need to establish some necessary a priori estimates on strong solutions to the Cauchy problem \eqref{NS}-\eqref{ydxw} for the extension of the local solutions. Thus, for the initial data $\left(\rho_0, u_0\right)$ satisfying \eqref{2ct}, let $(\rho, u)$ be the strong solution to the problem \eqref{NS}-\eqref{ydxw} on $\mathbb{R}^2 \times(0, T]$ in the class \eqref{2zzx} obtained in Lemma \ref{lem21}.

Now, we establish the following key estimates:
\begin{proposition}\la{vvvt} Under the conditions of Theorem \ref{T2},  for
\bnn
\tilde{E}_0\triangleq \rho_\infty+\|\rho_0-\rho_\infty\|_{L^r\cap L^\infty}+ \|\sqrt\n_0u_0 \|_{L^2}+ \|\na u_0\|_{L^2}+ \|\na \rho_0\|_{L^q},
\enn
there is a positive constant  $C $ depending only on $r$, $q$, $\underline\mu$, $\bar{\mu'}$ and $\tilde{E}_0$ such that
\be\la{bkj}
\sup_{t\in[0, T]}\|\nabla\rho\|_{L^q}+\int_0^T\|\na u\|_{L^\infty}dt\leq C(T).
\ee
\end{proposition}

Proposition \ref{vvvt} is an easy consequence of the following lemmas, whose proof is postponed to
later.
\subsection{\la{se99}A priori estimates (I): upper bound of the $\boldsymbol{\|\nabla\rho\|_{L^q}}$}
First, similar to Lemma \ref{lem3.1} and Lemma \ref{lem3.3}, their estimate do not involve the $L^p$ integrability of density, hence we have the following energy estimates directly:
\begin{lemma}\label{lem43.3}
There exists a positive constant $C$ depending only  on $\underline\mu, \bar{\mu'}, \|\rho_0\|_{L^\infty}, \|\sqrt{\n_0} u_0\|_{L^2}$ and $\|\na u_0\|_{L^2}$
such that for $i=0,1$,
\begin{align}
&\sup_{t\in[0,T]}\|\rho(t)-\rho_\infty\|_{L^r}\leq\|\rho_0-\rho_\infty\|_{L^r},\label{14.1}\\
&\sup_{t\in[0,T]} \|\sqrt{\rho} u \|_{L^2}^2 +\int_{0}^{T}\|\nabla u \|_{L^2}^{2}dt\leq C,\label{3.596}\\
&\sup_{t\in[0,T]}t^i\|\nabla u \|_{L^2}^2
+\int_{0}^{T}t^i\|\sqrt\rho\dot{ u }\|_{L^2}^2dt\leq C.\label{02223}
\end{align}
\end{lemma}

With appropriate modifications to Lemma \ref{lem3.4}, we have the following estimates:
\begin{lemma}\label{lem21.4}
There exists a positive constant $C$ depending only on $r$, $\rho_\infty$, $\underline\mu$, $\bar{\mu'}$, $\|\n_0-\rho_\infty\|_{L^r\cap L^\infty}$, $\|\sqrt{\n_0} u_0\|_{L^2}$ and $\|\na u_0\|_{L^2} $ such that for $i=1,2$,
\begin{equation}\label{133q}
\sup_{t\in[0,T]}t^i\|\sqrt{\rho}\dot{ u }\|_{L^2}^2
+\int_{0}^{T}t^i\|\nabla\dot{ u }\|_{L^2}^{2}dt\leq C,
\end{equation}
and
\begin{equation}\label{i313q}
\sup_{t\in[0,T]}\left(t^\frac{2(p-1)}{p}\|\nabla u\|_{L^p}^2+t^2\|\na\big(\mu(\rho)\o\big)\|_{L^2}^2+t^2\|\na P\|_{L^2}^2\right) \leq C(p).
\end{equation}
\end{lemma}
\begin{proof}
First, we denote $\psi\triangleq(\rho_\infty-\rho)_+$, by direct calculations we find for any $v\in D^1$ and $\sqrt\rho v\in L^2$,
\begin{equation}
\begin{split}\label{vzv1}
\rho_\infty\|v\|_{L^2}^2&\leq \|\sqrt\rho v\|_{L^2}^2+\|\sqrt \psi v\|_{L^2}^2\\
&\leq \|\sqrt\rho v\|_{L^2}^2+C\|\psi\|_{L^r}\|v\|_{L^2}^{2(r-1)/r}\|\nabla v\|_{L^2}^{2/r}\\
&\leq \|\sqrt\rho v\|_{L^2}^2+\frac{\rho_\infty}{2}\|v\|_{L^2}^2+C\|\nabla v\|_{L^2}^2,
\end{split}
\end{equation}
due to \eqref{14.1}. Which implies that
\begin{align}\label{zzxz}
\|v\|_{L^2}\leq C\|\sqrt\rho v\|_{L^2}+C\|\nabla v\|_{L^2}.
\end{align}

Next, it follows from Lemma \ref{stokes}, \eqref{NS}$_2$ and \eqref{14.1} that
\begin{align}\label{pazz}
\|\nabla\big(\mu(\rho)\o\big)\|_{L^2}+\|\nabla P\|_{L^2}\leq C\|\sqrt\rho\dot{u}\|_{L^2},
\end{align}
which together with Gagliardo-Nirenberg inequality and \eqref{dumm} shows that
\begin{equation}
\begin{split}\label{muy}
\|\nabla u\|_{L^4}\leq C\|\mu(\rho)\o\|_{L^4}&\leq C\|\nabla u\|_{L^2}^{1/2}\|\nabla\big(\mu(\rho)\o\big)\|_{L^2}^{1/2}\\
&\leq C\|\nabla u\|_{L^2}^{1/2}\|\sqrt\rho\dot{u}\|_{L^2}^{1/2}.
\end{split}
\end{equation}
This combined with \eqref{3.596} and \eqref{zzxz} leads to
\begin{align}\label{mmb}
\|u\|_{L^\infty}
&\leq C\|u\|_{L^2}^{1/3}\|\nabla u\|_{L^4}^{2/3}\nonumber\\
&\leq C(\|\sqrt\rho u\|_{L^2}+\|\nabla u\|_{L^2})^{1/3}\|\nabla u\|_{L^2}^{1/3}\|\sqrt\rho\dot{u}\|_{L^2}^{1/3}\\
&\leq C\|\nabla u\|_{L^2}^{1/3}\|\sqrt\rho \dot u\|_{L^2}^{1/3}.\nonumber
\end{align}

Then, it follows from \eqref{3.15} that
\begin{equation}\label{3.15q}
\begin{split}
\frac12\frac{d}{dt}\int\rho|\dot{ u }|^{2}dx&=\int\Big(\pl_t\pl_2\big(\mu(\rho)\omega\big)\dot{u}^1-\pl_t\pl_1\big(\mu(\rho)\omega\big)\dot{u}^2\Big) dx\\
&\quad+\int\Big( u\cdot\na\pl_2\big(\mu(\rho)\omega\big)\dot{u}^1-u\cdot\na\pl_1\big(\mu(\rho)\omega\big)\dot{u}^2\Big)dx\\
&\quad-\int\Big(\dot{u}^j \partial_{t}\partial_{j}P +\dot{u}^j u \cdot\nabla \partial_{j}P\Big) dx\triangleq\sum\limits_{i=1}^{3}\tilde{J}_i.
\end{split}
\end{equation}
It follows from integration by parts and combining with \eqref{plm}, \eqref{muy} that
$$
\tilde{J}_1+\tilde{J}_2\leq -C_1\|\nabla\dot{ u }\|_{L^2}^{2}+C\|\nabla u\|_{L^2}^2\|\sqrt\rho\dot{u}\|_{L^2}^2.
$$
Integration by parts together with \eqref{14.1}, \eqref{3.596} and \eqref{muy} gives
\begin{equation}
\begin{split}\nonumber
\tilde{J}_3&=-\int\big(u_t^j \partial_{t}\partial_{j}P+u^k\partial_k u^j \partial_{t}\partial_{j}P +\dot{u}^j u^k\partial_k \partial_{j}P\big) dx\\
&= -\frac{d}{dt}\int u^k\partial_k u^j \partial_{j}P dx+\int\big(\partial_{t}u^k\partial_k u^j \partial_{j}P+ u^k\partial_{t}\partial_k u^j \partial_{j}P\big)dx  \\
&\quad+\int\big(\partial_{j}\dot{u}^j u^k\partial_k P+\dot{u}^j \partial_{j}u^k\partial_k P\big)dx\\
&= -\frac{d}{dt}\int u^k\partial_k u^j \partial_{j}P dx+2\int \big(\dot u^k\partial_k u^j\partial_j P-u^i\partial_i u^k\partial_k u^j\partial_j P\big)dx\\
&\quad+\int\big(\partial_{j}u^i\partial_i u^j u^k\partial_k P+\dot{u}^j \partial_{j}u^k\partial_k P\big)dx\\
&\leq \frac{d}{dt}\int u^k\partial_k u^j \partial_{j}Pdx
+C\|\dot u\|_{BMO}\|\nabla u\|_{L^2}\|\nabla P\|_{L^2}+C\|u\|_{L^\infty}\|\nabla u\|_{L^4}^2\|\nabla P\|_{L^2}\\
&\leq \frac{d}{dt}\int u^k\partial_k u^j \partial_{j}Pdx+\varepsilon\|\nabla\dot{ u }\|_{L^2}^{2}+C\left(\|\nabla u\|_{L^2}^2+\|\nabla u\|_{L^2}^{4/3}\|\sqrt\rho\dot u\|_{L^2}^{1/3}\right)\|\sqrt\rho\dot u\|_{L^2}^2.
\end{split}
\end{equation}
Substituting $\tilde{J}_1$-$\tilde{J}_3$ into \eqref{3.15q}, and choosing $\varepsilon$ sufficiently small yield that
\be\ba\label{i47q}
\tilde{\Psi}'(t)+C_1\int|\nabla\dot{ u }|^{2}dx\le C\left(\|\nabla u\|_{L^2}^2+\|\nabla u\|_{L^2}^{4/3}\|\sqrt\rho\dot u\|_{L^2}^{1/3}\right)\|\sqrt\rho\dot u\|_{L^2}^2,
\ea\ee
where
\begin{equation*}
\tilde{\Psi}(t)\triangleq\frac12\int\rho|\dot{u}|^{2}dx
-\int u^k\partial_k u^j \partial_{j}Pdx,
\end{equation*}
satisfies
\be\la{psi1q}
\frac14\int\n |\dot u|^2dx-C\|\na u\|_{L^2}^4\le \tilde{\Psi}(t)\le \int\n |\dot u|^2dx+C\|\na u\|_{L^2}^4,
\ee
due to \eqref{mnxz}-\eqref{lem1} and \eqref{pazz}.

Multiplying \eqref{i47q} by $t$, and observe
\begin{equation}
\begin{split}\label{pop}
&\int_0^T t\|\nabla u\|_{L^2}^{4/3}\|\sqrt\rho\dot u\|_{L^2}^{1/3}\|\sqrt\rho\dot u\|_{L^2}^2dt\\
&\leq C\sup_{t\in[0,T]}(t\|\nabla u\|_{L^2}^2)^{2/3}\sup_{t\in[0,T]}(t\|\sqrt\rho\dot u\|_{L^2}^2)^{1/6}\int_0^T t^{1/6}\|\sqrt\rho\dot u\|_{L^2}^2dt\\
&\leq \varepsilon\sup_{t\in[0,T]}(t\|\sqrt\rho\dot u\|_{L^2}^2)^2+C,
\end{split}
\end{equation}
due to \eqref{3.596}-\eqref{02223}.
Choosing $\varepsilon$ sufficiently small, applying Gr\"onwall's inequality to \eqref{i47q} and using \eqref{3.596} and \eqref{psi1q}-\eqref{pop}, it holds that
$$
\sup_{t\in[0,T]}t\|\sqrt{\rho}\dot{ u }\|_{L^2}^2+\int_{0}^{T}t\|\nabla\dot{ u }\|_{L^2}^{2}dt\leq C.
$$

Multiplying \eqref{i47q} by $t^2$, and observe 
\begin{equation}
\begin{split}\label{popq}
&\int_0^T t^{2}\|\nabla u\|_{L^2}^{4/3}\|\sqrt\rho\dot u\|_{L^2}^{1/3}\|\sqrt\rho\dot u\|_{L^2}^2dt\\
&\leq C\sup_{t\in[0,T]}(t\|\nabla u\|_{L^2}^2)^{2/3}\sup_{t\in[0,T]}(t^{2}\|\sqrt\rho\dot u\|_{L^2}^2)^{1/6}\int_0^T t\|\sqrt\rho\dot u\|_{L^2}^2dt\\
&\leq \varepsilon\sup_{t\in[0,T]}(t^{2}\|\sqrt\rho\dot u\|_{L^2}^2)^2+C,
\end{split}
\end{equation}
due to \eqref{3.596}-\eqref{02223}.
Choose $\varepsilon$ is sufficiently small, applying Gr\"onwall's inequality to \eqref{i47q} and using \eqref{3.596}, \eqref{psi1q} and \eqref{popq}, it hold that
$$
\sup_{t\in[0,T]}t^2\|\sqrt{\rho}\dot{ u }\|_{L^2}^2+\int_{0}^{T}t^2\|\nabla\dot{ u }\|_{L^2}^{2}dt\leq C.
$$

Finally, it follows from \eqref{zxza} that for any $p\in(2,\infty)$,
\begin{align}
\|\nabla u\|_{L^p}\leq C\|\nabla u\|_{L^2}^\frac{2}{p}\|\sqrt\rho\dot u\|_{L^2}^\frac{p-2}{p},
\end{align}
which together with \eqref{02223}, \eqref{133q} and \eqref{pazz} yields \eqref{i313q}.
The proof of Lemma \ref{lem21.4} is finished.
\end{proof}

Now we are in a position to prove Proposition \ref{vvvt}.

{\it\bf Proof of Proposition \ref{vvvt}.}
It follows from H\"older inequality,  \eqref{14.1} and \eqref{zzxz} that
\begin{align}\label{4.6q}
\|\rho\dot{ u }\|_{L^q}
 &  \leq C\|\rho\dot{ u }\|_{L^2}^{\frac{2(q-1)}{q^2-2}}
\|\dot{ u }\|_{L^{q^2}}^{\frac{q(q-2)}{q^2-2}} \notag \\
 &  \leq C\|\rho\dot{ u }\|_{L^2}^{\frac{2(q-1)}{q^2-2}}
\left(\|\dot{ u }\|_{L^2}+\|\nabla\dot{ u }\|_{L^2}\right)^{\frac{q(q-2)}{q^2-2}}  \\
 &  \leq C\left(\|\sqrt{\rho}\dot{ u }\|_{L^2}
+\|\sqrt{\rho}\dot{ u }\|_{L^2}^{\frac{2(q-1)}{q^2-2}}
\|\nabla\dot{ u }\|_{L^2}^{\frac{q(q-2)}{q^2-2}}\right),\notag
\end{align}
which together with \eqref{3.596}-\eqref{133q} leads to
\begin{equation}
\begin{split}\label{zaq2}
&\int_{0}^{T}\|\rho\dot{ u }\|_{L^q}^{\frac{q+1}{q}}dt \\
&\leq C\int_{0}^{1}\left(\|\sqrt{\rho}\dot{ u }\|_{L^2}^\frac{q+1}{q}
+\|\sqrt{\rho}\dot{ u }\|_{L^2}^{\frac{2(q^2-1)}{q(q^2-2)}}
\|\nabla\dot{ u }\|_{L^2}^{\frac{(q^2-q-2)}{q^2-2}}\right)dt\\
&\quad+C\int_{1}^{T}\left(\|\sqrt{\rho}\dot{ u }\|_{L^2}^\frac{q+1}{q}
+\|\sqrt{\rho}\dot{ u }\|_{L^2}^{\frac{2(q^2-1)}{q(q^2-2)}}
\|\nabla\dot{ u }\|_{L^2}^{\frac{(q^2-q-2)}{q^2-2}}\right)dt  \\
&  \leq C+\left(\int_0^1 \|\sqrt{\rho}\dot{ u }\|_{L^2}^2dt\right)^{\frac{(q^2-1)}{q(q^2-2)}}\left(\int_0^1t\|\nabla\dot{ u }\|_{L^2}^2dt\right)^{\frac{(q^2-q-2)}{2(q^2-2)}}\left(\int_0^1t^{-\frac{(q^2-q-2)q}{(q-1)(q^2-2)}} dt\right)^\frac{(q-1)}{2q}\\
&\quad+C\sup_{t\in[1,T]}(t^2\|\sqrt{\rho}\dot{ u }\|_{L^2}^2)^\frac{q+1}{2q}\int_{1}^{T}t^{-\frac{q+1}{q}}dt\\
&\quad+\left(\int_1^T t\|\sqrt{\rho}\dot{ u }\|_{L^2}^2dt\right)^{\frac{(q^2-1)}{q(q^2-2)}}\left(\int_1^Tt^2\|\nabla\dot{ u }\|_{L^2}^2dt\right)^{\frac{(q^2-q-2)}{2(q^2-2)}}\left(\int_1^Tt^{-\frac{2(q^3-2q-1)}{(q-1)(q^2-2)}} dt\right)^\frac{(q-1)}{2q}\\
&\leq C.
\end{split}
\end{equation}

Similar to the proof of Proposition \ref{aupper}, we then obtain
$$\sup_{t\in[0, T]}\|\nabla\rho\|_{L^q}+\int_0^T\|\na u\|_{L^\infty}dt\leq C(T).$$
We complete the proof of Proposition \ref{vvvt}.
\subsection{\la{se14}A priori estimates (II): higher order estimates}
Similar to the proof of Lemma \ref{lem3.5}, we also obtain following estimates:
\begin{lemma}\label{lem77.5}
There exists a positive constant $C$ depending only on $r$, $q$, $\rho_\infty$, $\underline\mu$, $\bar{\mu'}$, $\|\n_0-\rho_\infty\|_{L^r\cap L^\infty\cap D^{1,q}}$, $\|\sqrt{\n_0} u_0\|_{L^2}$ and $\|\na u_0\|_{L^2} $ such that
\begin{equation}
\begin{split}\label{16.1q}
&\int_{0}^{T}\left(\|\nabla^{2} u \|_{L^2}^2+\|\nabla^{2} u \|_{L^q}^{\frac{q+1}{q}}
+t\|\nabla^{2} u \|_{L^2 \cap L^q}^2\right)dt \\
&+\int_{0}^{T}\left(\|\nabla P\|_{L^2}^2+\|\nabla P \|_{L^q}^\frac{q+1}{q}+t\|\nabla P\|_{L^2\cap L^q}^2\right)dt
\leq C(T).
\end{split}
\end{equation}
\end{lemma}
\begin{proof}
It from \eqref{4.6q}, we have
\begin{equation}
\begin{split}\nonumber
 &  \int_{0}^{T}t\|\rho\dot{ u }\|_{L^q}^{2}dt \\
 &  \leq C\int_{0}^{T}\left(t\|\sqrt{\rho}\dot{ u }\|_{L^2}^2
+t\|\sqrt{\rho}\dot{ u }\|_{L^2}^{\frac{4(q-1)}{q^2-2}}
\|\nabla\dot{ u }\|_{L^2}^{\frac{2q(q-2)}{q^2-2}}\right)dt  \\
 &  \leq C+C\left(\int_0^Tt\|\sqrt{\rho}\dot{ u }\|_{L^2}^2dt\right)^{\frac{2(q-1)}{q^2-2}}\left(\int_0^Tt\|\nabla\dot{ u }\|_{L^2}^2
 dt\right)^{\frac{q(q-2)}{q^2-2}}\\
 &\leq C,
\end{split}
\end{equation}
due to \eqref{02223} and \eqref{133q}.

Which combination of \eqref{2dumqa}, \eqref{paztg}, \eqref{3.596}, \eqref{pazz} and \eqref{zaq2} thus \eqref{16.1q} and completes the proof of Lemma \ref{lem77.5}.
\end{proof}

With appropriate modifications to Lemma \ref{lem3.8}, we have the following higher estimates:
\begin{lemma}\label{lem3.8a}
There exists a positive constant $C$ depending only on $r$, $q$, $\rho_\infty$, $\underline\mu$, $\bar{\mu'}$, $\|\n_0-\rho_\infty\|_{L^r\cap L^\infty\cap D^{1,q}}$, $\|\sqrt{\n_0} u_0\|_{L^2}$ and $\|\na u_0\|_{L^2} $ such that
\begin{equation}\label{7.1a}
\sup_{t\in[0,T]}t\|\sqrt{\rho}u_t \|_{L^2}^{2}
+\int_{0}^{T}t\|\nabla u_t \|_{L^2}^2dt\leq C(T).
\end{equation}
\end{lemma}
\begin{proof}
It follows from \eqref{7.3} that
\begin{equation}
\begin{split}\label{7.3q}
&\frac{1}{2}\frac{d}{dt}\|\sqrt{\rho} u_t \|_{L^2}^{2}
+\int\mu(\rho)|\o_t|^2dx\\
& =-\int \rho_{t}|u_t |^2dx
-\int (\rho u )_{t}\cdot\nabla u \cdot u_t dx-\int \nabla^\perp\big(u\cdot\nabla\mu(\rho)\o\big)\cdot u_tdx\\
&  \leq C\int\rho| u ||u_t |(|\nabla u_t|+|\nabla u |^2
+| u ||\nabla^2 u |)dx  \\
& \quad +C\int\rho| u |^2|\nabla u||\nabla u_t|dx+C\int\rho|u_t|^2|\nabla u |dx+C\int|u\cdot\nabla\rho||\omega\omega_t |dx  \\
& \triangleq \tilde K_1+\tilde K_2+\tilde K_3+\tilde K_4.
\end{split}
\end{equation}

Now, we will use the Garliardo-Nirenberg inequality, \eqref{bkj}-\eqref{3.596}, \eqref{muy} and \eqref{mmb} to estimate each term on the right-hand side of \eqref{7.3q} as follows:
$$
\begin{aligned}\label{7.5q}
\tilde K_1 &  \leq C\|\rho\|_{L^\infty}\|u \|_{L^4}\|u_t\|_{L^4}
\left(\|\nabla u_t\|_{L^2}+\|\nabla u \|_{L^4}^2\right) \notag \\
 & \quad +C\|\rho\|_{L^\infty}\| u \|_{L^{8}}^{2}\|u_t\|_{L^4}\|\nabla^2 u \|_{L^2}\notag \\
 &  \leq C(\|\sqrt{\rho}u_t\|_{L^2}+\|\nabla u_t\|_{L^2})^{1/2}\|\nabla u_t\|_{L^2}^{1/2}(\| \na u_{t}\|_{L^{2}} +\|\nabla^2 u\|_{L^2})
\notag \\
&  \leq  \frac{1}{6}\|\sqrt{\mu(\rho)}\o_t\|_{L^2}^{2}+C\|\sqrt{\rho}u_t\|_{L^2}^2+C\|\nabla^2 u\|_{L^2}^2+C,
\end{aligned}
$$
$$
\begin{aligned}\label{7.6q}
\tilde K_2+\tilde K_3
&  \leq C \|\rho\|_{L^\infty}\| u\|_{L^8}^{2}\|\na u\|_{L^4} \| \na u_{t}\|_{L^{2}} +C\|\nabla u \|_{L^2}\| \sqrt\rho u_t\|_{L^6}^{3/2}\| \sqrt\rho u_t\|_{L^2}^{1/2}\notag \\
&  \leq \frac{1}{6}\|\sqrt{\mu(\rho)}\o_t\|_{L^2}^{2}+C\|\sqrt{\rho}u_t\|_{L^2}^2+C\|\nabla^2 u\|_{L^2}^2+C,
\end{aligned}
$$
and
$$
\begin{aligned}
\tilde K_4 &  \leq C\|u \|_{L^{4q/(q-2)}}\|\na\rho\|_{L^q}
\|\o\|_{L^{4q/(q-2)}}\|\o_t\|_{L^2} \notag \\
&  \leq C(1+\|\nabla^2 u\|_{L^2})\|\o_t\|_{L^2}\notag \\
&  \leq  \frac{1}{6}\|\sqrt{\mu(\rho)}\o_t\|_{L^2}^{2}+C\|\nabla^2 u\|_{L^2}^2+C.
\end{aligned}
$$

Substituting $\tilde K_1$-$\tilde K_4$ into \eqref{7.3q} gives
\begin{equation}\label{11.3q}
\frac{d}{dt}\|\sqrt{\rho}u_t\|_{L^2}^{2}
+\int\mu(\rho)|\o_t|^2dx\leq C\|\sqrt{\rho}u_t\|_{L^2}^2+C\|\nabla^2 u\|_{L^2}^2+C.
\end{equation}

Finally, noticing that it follows from \eqref{14.1}-\eqref{02223}, \eqref{zzxz} and Garliardo-Nirenberg inequality that
\begin{equation}
\begin{split}\label{jwaq}
\int_0^T\|\sqrt{\rho} u_t\|_{L^2}^2dt&\le \int_0^T\left(\|\sqrt{\rho}\dot u\|_{L^2}^2+\|\sqrt{\rho}|u||\na u|\|_{L^2}^2\right)dt\\
&\le C+C\int_0^T\|\rho\|_{L^\infty}\| u \|_{L^6}^2\|\na u\|_{L^3}^2dt \\
&\le C+C\int_0^T\left(\|\nabla u\|_{L^2}^2+\|\nabla^2 u\|_{L^2}^2\right)dt\\
&\le C,
\end{split}
\end{equation}
multiplying \eqref{11.3q} by $t$, using Gr\"onwall's inequality, \eqref{16.1q} and \eqref{jwaq}, we obtain \eqref{7.1a} and complete the proof of Lemma \ref{lem3.8a}.
\end{proof}
\section{Proof of Theorems \ref{T1} and \ref{T2}}\label{sec4}
With all the priori estimates in Sections \ref{se13} and \ref{se321} at hand, it is time to prove Theorems \ref{T1} and \ref{T2}. In fact, since the proofs of two theorems are totally similar, we mainly prove Theorem \ref{T1} for simplicity.

By Lemma 2.1, there exists a $T_{*}>0$ such that the Cauchy problem \eqref{NS}-\eqref{ydxw} has a unique local strong solution $(\rho,u,P)$ on $\mathbb{R}^2\times(0,T_{*}]$. We plan to extend the local solution to all time.

Set
\begin{equation}\label{20.1}
T^{*}=\sup \{T~|~(\rho,u,P)\ \text{is a strong solution on}\ \mathbb{R}^{2}\times(0,T]\}.
\end{equation}
First, for any $0<\tau<T_*<T\leq T^{*}$ with $T$ finite, one deduces from \eqref{3.5}, \eqref{i313} and
\eqref{7.1} that for any $q>2$,
\begin{equation}\label{20.2}
\nabla u \in C([\tau,T];L^2\cap L^q),
\end{equation}
where one has used the standard embedding
\begin{equation*}
L^{\infty}(\tau,T;H^1)\cap H^{1}(\tau,T;H^{-1})\hookrightarrow C(\tau,T;L^q)\ \ \text{for any}\ \ q\in(2,\infty).
\end{equation*}
Moreover, it follows from \eqref{16.1}, \eqref{6.1}, and \cite[Lemma 2.3]{L1996} that
\begin{equation}\label{20.3}
\rho\in C([0,T];L^1\cap H^1\cap W^{1,q}).
\end{equation}
Thanks to \eqref{0.11}-\eqref{3.5} and \eqref{6.1}, the standard argument yield that
$$
\rho |u|^2\in H^1(0,T;L^1)\hookrightarrow C(0,T;L^1).
$$

Finally, if $T^{*}<\infty,$ it follows from  \eqref{0.11}-\eqref{3.5}, \eqref{6.1} and \eqref{20.2}-\eqref{20.3} that
$$
(\n, u)(x,T^*)=\lim_{t\rightarrow T^*}(\n, u)(x,t),
$$
satisfies the initial conditions \eqref{ct} at $t=T^*$. Thus, taking $(\n, u)(x,T^*)$ as the initial data, Lemma 2.1 implies that one can extend the  strong solutions beyond $T^*$. This contradicts the assumption of $T^*$ in \eqref{20.1}. The proof of Theorem \ref{T1} is completed.  \hfill $\Box$


\begin{thebibliography}{10}
\bibitem{Abi2021}
H. Abidi, G.L. Gui, Global well-posedness for the 2-D inhomogeneous incompressible Navier-Stokes system with large initial data in critical spaces, Arch. Ration. Mech. Anal. 242 (2021) 1533-1570.

\bibitem{Abi2011}
H. Abidi, G.L. Gui, P. Zhang, On the decay and stability to global solutions of the 3-D inhomogeneous Navier-Stokes equations, Commun. Pure. Appl. Math. 64 (2011) 832-881.

\bibitem{zhang5}
H. Abidi, P. Zhang, On the global well-posedness of 2-D density-dependent Navier-Stokes system with variable
viscosity, J. Differ. Equ. 259 (2015) 3755-3802.

\bibitem{zhang2015}
H. Abidi, P. Zhang, Global well-posedness of 3-D density-dependent Navier-Stokes system with variable viscosity,
Sci. China Math. 58 (6)  (2015) 1129-1150.

\bibitem{AKM1990}
S.A. Antontesv, A.V. Kazhikov, V.N. Monakhov,
Boundary Value Problems in Mechanics of Nonhomogeneous Fluids,
North-Holland, Amsterdam, 1990.

\bibitem{B1}
J.T. Beale, T. Kato, A. Majda,
Remarks on the breakdown of smooth solutions for the 3-D Euler
equations, Commun. Math. Phys. 94 (1984) 61-66.

\bibitem{CK2003}
H.J. Choe, H. Kim,
Strong solutions of the Navier-Stokes equations for nonhomogeneous  incompressible fluids,
Comm. Partial Differential Equations. 28 (2003) 1183--1201.


\bibitem{CLMS1993}
R. Coifman, P.L. Lions, Y. Meyer, S. Semmes,
Compensated compactness and Hardy spaces,
J. Math. Pures Appl. 72 (1993) 247-286.

\bibitem{craig2013global}
W. Craig, X. Huang, and Y. Wang, Global well-posedness for the 3d inhomogeneous incompressible navier-stokes equations, Journal of Mathematical Fluid Mechanics, 15 (2013), pp. 747-758.

\bibitem{D1997}
B. Desjardins,
Regularity results for two-dimensional flows of multiphase viscous fluids,
Arch. Rational. Mech. Anal. 137 (1997) 135-158.

\bibitem{Fan2022}
X.Y. Fan, J.X. Li, J. Li, Global existence of strong and weak solutions to 2D compressible Navier-Stokes system in
bounded domains with large data and vacuum, Arch. Ration. Mech. Anal. 245 (1) (2022) 239-278.

\bibitem{Fe}
E. Feireisl, Dynamics of viscous compressible fluids, Oxford University Press, 2004.

\bibitem{He2021}
C. He, J. Li, B.Q. L\"u,
Global well-posedness and exponential stability of 3D Navier Stokes equations with density-dependent viscosity and vacuum in unbounded domains,
Arch. Rational. Mech. Anal. 239 (2021) 1809-1835.

\bibitem{Hua2016}
X.D. Huang, J. Li, Existence and blowup behavior of global strong solutions to the two dimensional barotropic
compressible Navier-Stokes system with vacuum and large initial data, J. Math. Pures. Appl. 106 (9) (2016) 123-154.

\bibitem{Hua2022}
X.D. Huang, J. Li, Global well-posedness of classical solutions to the Cauchy problem of two-dimensional
barotropic compressible Navier-Stokes system with vacuum and large initial data, SIAM J. Math. Anal. 54 (3)
(2022) 3192-3214.

\bibitem{H-l-r}
X.D. Huang, J.X. Li, R. Zhang, Existence and exponential stability of the global large strong solution to the 3D inhomogeneous Navier-Stokes equations with density-dependent viscosity and large velocity, Calc. Var. 65, 56 (2026).


\bibitem{HW2014}
X.D. Huang, Y. Wang,
Global strong solution with vacuum to the two-dimensional density-dependent Navier-Stokes system,
SIAM J. Math. Appl. 46 (2014) 1771-1788.

\bibitem{HW2015}
X.D. Huang, Y. Wang,
Global strong solution of 3D inhomogeneous Navier-Stokes equations with density-dependent viscosity. J. Differ. Equ. 259 (2015) 1606-1627.

\bibitem{kato}
T. Kato, Remarks on the Euler and Navier-Stokes equations in $\mathbb{R}^2$, Proc. Symp. Pure Math. Vol. 45, Amer. Math. Soc. Providence, (1986) 1-7.

\bibitem{K1974}
A.V. Kazhikov,
Resolution of boundary value problems for nonhomogeneous viscous fluids,
Dokl. Akad. Nauk. 216 (1974) 1008-1010.


\bibitem{L1996}
P.L. Lions,
Mathematical topics in fluid mechanics, vol. 1. Incompressible Models,
Oxford University Press, New York, 1996.

\bibitem{Liu2019}
X.F. Liu,
Global well-posedness to some modified 2-D nonhomogeneous Navier-Stokes equations, J. Math. Phys. 60 (2019) 081510.

\bibitem{LZ2015}
B.Q. L\"u, Z.H. Xu, X. Zhong,
On local strong solutions to the Cauchy problem of the two-dimensional density-dependent magnetohydrodynamic equations with vacuum, arXiv:1506.02156, 2015.

\bibitem{Lv2018}
B.Q. L\"u, X.D. Shi, X. Zhong,
Global existence and large time asymptotic behavior of strong solutions to the Cauchy problem of 2D density-dependent Navier Stokes-equations with vacuum, Nonlinearity. 31 (2018) 2617-2632.


\bibitem{N1959}
L. Nirenberg,
On elliptic partial differential equations,
Ann. Scuola Norm. Sup. Pisa. 13 (1959) 115-162.

\bibitem{pa2020}
M. Paicu, P. Zhang,
Striated Regularity of 2-D Inhomogeneous Incompressible Navier-Stokes System with Variable Viscosity, Commun. Math. Phys. 376 (2020) 385-439.

\bibitem{S1990}
J. Simon,
Nonhomogeneous viscous incompressible fluids: existence of velocity, density, and pressure,
SIAM J. Math. Anal. 21 (1990) 1093-1117.


\bibitem{S1993}
E.M. Stein,
Harmonic analysis: real-variable methods, orthogonality, and oscillatory integrals,
Princeton University Press, Princeton, NJ 1993.


\bibitem{Z2015}
J.W. Zhang, Global well-posedness for the incompressible Navier-Stokes equations with density-dependent viscosity coefficient, J. Differ. Equ. 259 (2015) 1722-1742.

\bibitem{Z2008}
P. Zhang, Global smooth solutions to the 2-D nonhomogeneous Navier-Stokes equations, Int. Math. Res. Not. (2008) rnn098.
\end{thebibliography}
\end{document}